\documentclass[sn-apa]{sn-jnl}

\usepackage{microtype}
\usepackage{verbatim}
\usepackage{times}
\usepackage{mathtools}
\usepackage{caption}
\usepackage{subcaption}

\usepackage[enable]{easy-todo}
\usepackage{color,soul}

\newcommand{\cK}{\mathcal{K}}
\newcommand{\Rd}{\mathbb{R}^d}
\newcommand{\abs}[1]{\vert #1 \rvert}
\newcommand{\bd}[3]{\mathcal{B}_{#1}(#2,#3)}
\newcommand{\scf}{\phi}
\newcommand{\dualscf}{\phi^{*}}
\newcommand{\sct}[1]{\phi_{#1}}
\newcommand{\dualsct}[1]{\phi^*_{#1}}
\newcommand{\mscf}{\Phi}
\newcommand{\dualmscf}{\Phi^*}
\newcommand{\msct}[1]{\Phi_{#1}}
\newcommand{\dualmsct}[1]{\Phi^{*}_{#1}}
\newcommand{\scfi}{\phi}

\newcommand{\comp}{r}
\newcommand{\sgn}{\operatorname{sgn}}
\newcommand{\grad}{\triangledown }
\newcommand{\norm}[1]{\lVert #1\rVert }
\newcommand{\dualnorm}[1]{\lVert #1\rVert_* }
\newcommand{\inner}[2]{\langle #1,#2\rangle }
\newcommand{\ex}[1]{\mathbb{E}[#1]}
\newcommand{\tr}{\operatorname{Tr}}
\newcommand{\diag}{\operatorname{diag}}

\newcommand{\sequ}[1]{\{#1_t\}}

\DeclareMathOperator*{\argmin}{\arg\min} 

\jyear{2022}%

\theoremstyle{thmstyleone}%
\newtheorem{theorem}{Theorem}
\newtheorem{proposition}{Proposition}%
\newtheorem{lemma}{Lemma}%
\newtheorem{corollary}{Corollary}

\theoremstyle{thmstyletwo}%

\theoremstyle{thmstylethree}%

\begin{document}

\title[Optimistic Exponentiated Update]{Optimistic Optimisation of Composite Objective with Exponentiated Update}


\author*[1]{\sur{Weijia} \fnm{Shao}} \email{weijia.shao@campus.tu-berlin.de}

\author[2]{\sur{Fikret} \fnm{Sivrikaya}}  \email{fikret.sivrikaya@gt-arc.com}

\author[1,2]{\sur{Sahin} \fnm{Albayrak}} \email{sahin.albayrak@dai-labor.de}

\affil*[1]{\orgdiv{Faculty of Electrical Engineering and Computer Science}, \orgname{Technische Universit\"at Berlin}, \orgaddress{\street{Ernst-Reuter-Platz 7}, \city{Berlin}, \postcode{10587}, \country{Germany}}}

\affil[2]{\orgname{ GT-ARC Gemeinn\"utzige GmbH}, \orgaddress{\street{Ernst-Reuter-Platz 7}, \city{Berlin}, \postcode{10587}, \country{Germany}}}

\abstract{This paper proposes a new family of algorithms for the online optimisation of composite objectives. The algorithms can be interpreted as the combination of the exponentiated gradient and $p$-norm algorithm. Combined with algorithmic ideas of adaptivity and optimism, the proposed algorithms achieve a sequence-dependent regret upper bound, matching the best-known bounds for sparse target decision variables. Furthermore, the algorithms have efficient implementations for popular composite objectives and constraints and can be converted to stochastic optimisation algorithms with the optimal accelerated rate for smooth objectives.}




\maketitle
\section{Introduction}
Many machine learning problems involve minimising high dimensional composite objectives \citep{lu2014smoothed,ribeiro2016should,NEURIPS2018_c5ff2543,xie2018nonstop}. For example, in the task of explaining predictions of an image classifier \citep{ribeiro2016should,NEURIPS2018_c5ff2543}, we need to find a sufficiently small set of features explaining the prediction by solving the following constrained optimisation problem
\[
\begin{split}
\min_{x\in\Rd}\quad &l(x)+\lambda_1\norm{x}_1+\frac{\lambda_2}{2}\norm{x}_2^2\\\
\textrm{s.t.} \quad &\abs{x_i}\leq c_i \textrm{ for all } i=1,\ldots, d,  \\
\end{split}
\]
where $l$ is a function relating to the classifier, $\lambda_1$ controls the sparsity of the feature set, $\lambda_2$ controls the complexity of the feature set, and $c_1,\ldots,c_d$ are the ranges of the features. For $l$ with a complicated structure and large $d$, it is practical to solve the problem by optimising the first-order approximation of the objective function \citep{lan2020first}. However, the first-order methods can not attain optimal performance due to the non-smooth component $\lambda_1\norm{\cdot}_1$. Furthermore, the purpose of introducing the $\ell_1$ regularisation is to ensure the sparsity of the decision variable. Applying the first-order algorithms directly on the subgradient of $\lambda_1\norm{\cdot}_1$ does not lead to sparse updates \citep{duchi2010composite}. We refer to the objective function consisting of a loss with a complicated structure and a simple (possibly non-smooth) convex regularisation term as a composite objective. 

This paper focuses on the more general online convex optimisation (\textbf{OCO}), which can be considered as an iterative game between a player and an adversary. In each round $t$ of the game, the player makes a decision $x_t\in \cK$. Next, the adversary selects and reveals a convex loss $l_t$ to the player, who then suffers the composite loss $f_t(x)=l_t(x)+\comp_t(x)$, where $l_t:\cK\to \mathbb{R}$ is a convex function revealed at each iteration and $\comp_t:\mathbb{X}\to \mathbb{R}_{\geq 0}$ is a known closed convex function. The target is to develop algorithms minimising the regret of not choosing the best decision $x\in\cK$
\[
\mathcal{R}_{1:T}=\sum_{t=1}^Tf_t(x_t)-\min_{x\in\cK}\sum_{t=1}^Tf_t(x).
\]
An online optimisation algorithm can be converted into a stochastic optimisation algorithm using the online-to-batch conversion technique \citep{cesa2004generalization}, which is our primary motivation. In addition to that, online optimisation also has many direct applications, such as recommender systems \citep{song2014online} and time series prediction \citep{anava2013online}.

Given a sequence of subgradients $\sequ{g}$ of $\sequ{l}$, we are interested in the so-called adaptive algorithms ensuring regret bounds of the form $\mathcal{O}(\sqrt{\sum_{t=1}^T\dualnorm{g_t}^2})$. The adaptive algorithms are worst-case optimal in the online setting \citep{mcmahanadaptive} and can be converted into stochastic optimisation algorithms with optimal convergence rates \citep{levy2018online,kavis2019unixgrad,cutkosky2019anytime, joulani2020simpler}. The adaptive subgradient methods (\textbf{AdaGrad}) \citep{duchi2011adaptive} and their variants \citep{duchi2011adaptive,orabona2015generalized,orabona2018scale,alacaoglu2020new} have become the most popular adaptive algorithms in recent years. They are often applied to estimating deep learning models and outperform standard optimisation algorithms when the gradient vectors are sparse. However, such property can not be expected in every problem. If the decision variables are in an $\ell_1$ ball and gradient vectors are dense, the \textbf{Adagrad}-style algorithms do not have an optimal theoretical guarantee due to the sub-linear regret dependence on the dimensionality. 

The exponentiated gradient (\textbf{EG}) methods \citep{kivinen1997exponentiated,arora2012multiplicative}, which are designed for estimating weights in the positive orthant, enjoy the regret bound growing logarithmically with the dimensionality. The $\operatorname{\textbf{EG}^\pm}$ algorithm generalises this idea to negative weights \citep{kivinen1997exponentiated,warmuth2007winnowing}. Given $d$ dimensional problems with the maximum norm of the gradient bounded by $G$, the regret of $\operatorname{\textbf{EG}^\pm}$ is upper bounded by $\mathcal{O}(G\sqrt{T\ln d})$. As the performance of the $\operatorname{\textbf{EG}^\pm}$ algorithm depends strongly on the choice of hyperparameters, the $p$-norm algorithm \citep{gentile2003robustness}, which is less sensitive to the tuning of hyperparameters, is introduced to approach the logarithmic behaviour of $\operatorname{\textbf{EG}^\pm}$. \cite{kakade2012regularization} further extends the $p$-norm algorithm to learning with matrices. An adaptive version of the $p$-norm algorithm is analysed in \cite{orabona2015generalized}, which has a regret upper bound proportional to $\norm{x}_{p,*}^2\sqrt{\sum_{t=1}^T\norm{g_t}_{p}^2}$ for a given sequence of gradients $\sequ{g}$. By choosing $p=2\ln d$, a regret upper bound $\mathcal{O}(\norm{x}_1^2\sqrt{\ln d \sum_{t=1}^T\norm{g_t}_\infty^2})$ can be achieved. However, tuning hyperparameters is still required to attain the optimal regret $\mathcal{O}(\norm{x}_1\sqrt{\ln d \sum_{t=1}^T\norm{g_t}_\infty^2})$.

Recently, \cite{ghai2020exponentiated} has introduced a hyperbolic regulariser for online mirror descent update (\textbf{HU}), which can be viewed as an interpolation between gradient descent and \textbf{EG}. It has a logarithmic behaviour as in \textbf{EG} and a stepsize that can be flexibly scheduled as gradient descent. However, many optimisation problems with sparse targets have an $\ell_1$ or nuclear regulariser in the objective function. Otherwise, the optimisation algorithm has to pick a decision variable from a compact decision set. Due to the hyperbolic regulariser, it is difficult to derive a closed-form solution for either case. \cite{ghai2020exponentiated} has proposed a workaround by tuning a temperature-like hyperparameter to normalise the decision variable at each iteration, which is equivalent to the $\operatorname{\textbf{EG}^\pm}$ algorithm and leads to a performance dependence on the tuning. 

This paper proposes a family of algorithms for the online optimisation of composite objectives. The algorithms employ an entropy-like regulariser combined with algorithmic ideas of adaptivity and optimism. Equipped with the regulariser, the online mirror descent (\textbf{OMD}) and the follow-the-regulariser-leader (\textbf{FTRL}) algorithms update the absolute value of the scalar components of the decision variable in the same way as \textbf{EG} in the positive orthant. The directions of the decision variables are set in the same way as the $p$-norm algorithm. To derive the regret upper bound, we first show that the regulariser is strongly convex with respect to the $\ell_1$-norm over the $\ell_1$ ball. Then we analyse the algorithms in the comprehensive framework for optimistic algorithms with adaptive regularisers \citep{joulani2017modular}. Given the radius of decision set $D$, sequences of gradients $\sequ{g}$ and hints $\sequ{h}$, the proposed algorithms achieve a regret upper bound in the form of $\mathcal{O}(D\sqrt{\ln d\sum_{t=1}^T\norm{g_t-h_t}^2_\infty})$. With the techniques introduced in \cite{ghai2020exponentiated}, a spectral analogue of the entropy-like regulariser can be found and proved to be strongly convex with respect to the nuclear norm over the nuclear ball, from which the best-known regret upper bound depending on $\sqrt{\ln(\min\{m,n\})}$ for problems in $\mathbb{R}^{m,n}$ follows. 

Furthermore, the algorithms have closed-form solutions for the $\ell_1$ and nuclear regularised objective functions. For the $\ell_2$ and Frobenius regularised objectives, the update rules involve values of the principal branch of the \textit{Lambert function}, which can be well approximated. We propose a sorting based procedure projecting the solution to the decision set for the $\ell_1$ or nuclear ball constrained problems. Finally, the proposed online algorithms can be converted into algorithms for stochastic optimisation with the technique introduced in \cite{joulani2020simpler}. We show that the converted algorithms guarantee an optimal accelerated convergence rate for smooth objective functions. The convergence rate depends logarithmically on the dimensionality of the problem, which suggests its advantage compared to the accelerated \textbf{AdaGrad}-Style algorithms \citep{levy2018online,cutkosky2019anytime,joulani2020simpler}. 

The rest of the paper is organised as follows. Section~2 reviews the existing work. Section~3 introduces the notation and preliminary concepts. Next, we present and analyse our algorithms in Section~4. In Section~5, we derive efficient implementations for some popular choices of composite objectives, constraints and stochastic optimisation. Section~6 demonstrates the empirical evaluations using both synthetic and real-world data. Finally, we conclude our work in Section~7.

\section{Related Work}
Our primary motivation is to solve the optimisation problems with an elastic net regulariser in their objective function, which are highly involved in attacking \citep{carlini2017towards,chen2018ead,9413170} and explaining \citep{NEURIPS2018_c5ff2543,ribeiro2016should} deep neural networks. The proximal gradient method (\textbf{PGD}) \citep{nesterov2003introductory} and its accelerated variants \citep{beck2009fast} are usually applied to solving the problem. However, these algorithms are not practical since they require prior knowledge about the smoothness of the objective function to ensure their convergence. 

The \textbf{AdaGrad}-style algorithms \citep{duchi2011adaptive,orabona2015generalized,orabona2018scale,alacaoglu2020new} have become popular in the machine learning community in recent years. Given the gradient vectors $g_1,\ldots, g_t$ received at iteration $t$, the core idea of these algorithms is to set the stepsizes proportional to $\frac{1}{\sqrt{\sum_{s=1}^{t-1}\dualnorm{g_s}^2}}$ to ensure a regret upper bounded by $\mathcal{O}(\sqrt{\sum_{t=1}^{T}\dualnorm{g_t}^2})$ after $T$ iterations. Online learning algorithms with this adaptive regret can be directly applied to the stochastic optimisation problems \citep{li2019convergence,alacaoglu2020new} or can be converted into a stochastic algorithm \citep{cesa2008improved} with a convergence rate $\mathcal{O}(\frac{1}{\sqrt{T}})$. This rate can be further improved to $\mathcal{O}(\frac{1}{T^2})$ for unconstrained problems with smooth loss functions by applying the acceleration techniques \citep{levy2018online,kavis2019unixgrad,cutkosky2019anytime}. These acceleration techniques do not require prior knowledge about the smoothness of the loss function and a guarantee convergence rate of $\mathcal{O}(\frac{1}{\sqrt{T}})$ for non-smooth functions. \cite{joulani2020simpler} has proposed a simple approach to accelerate optimistic online optimisation algorithms with adaptive regret bound.

Given a $d$-dimensional problem, the algorithms mentioned above have a regret upper bound depending (sub-) linearly on $d$. We are interested in a logarithmic regret dependence on the dimensionality, which can be attained by the $\textbf{EG}$ family algorithms \citep{kivinen1997exponentiated,warmuth2007winnowing,arora2012multiplicative} and their adaptive optimistic extension \citep{steinhardt2014adaptivity}. However, these algorithms work only for decision sets in the form of cross-polytopes and require prior knowledge about the radius of the decision set for general convex optimisation problems. The $p$-norm algorithm \citep{gentile2003robustness,kakade2012regularization} does not have the limitation mentioned above; however, it still requires prior knowledge about the problem to attain optimal performance \citep{orabona2015generalized}. The \textbf{HU} algorithm \citep{ghai2020exponentiated}, which interpolates gradient descent and \textbf{EG}, can theoretically be applied to loss functions with elastic net regularisers and decision sets other than cross-polytopes. However, it is not practical due to the complex projection step. 

Following the idea of \textbf{HU}, we propose more practical algorithms interpolating \textbf{EG} and the $p$-norm algorithm. The core of our algorithm is a symmetric logarithmic function. \cite{orabona2013dimension} first introduced the idea of composing the single-dimensional symmetric logarithmic function and a norm to generalise \textbf{EG} to the infinite-dimensional space. It has become popular for parameter-free optimisation \citep{cutkosky2016online,cutkosky2017online,cutkosky2017stochastic,kempka2019adaptive} since one can easily construct an adaptive regulariser with this composition \citep{cutkosky2017online}. In this paper, instead of using the composition, we apply the symmetric logarithmic function directly to each entry of a vector to construct a symmetric entropy-like function that is strongly convex with respect to the $\ell_1$ norm. We analyse \textbf{MD} and \textbf{FTRL} with the entropy-like function in the framework developed in \cite{joulani2017modular}. The analysis of the spectral analogue of the entropy-like function follows the idea proposed in \cite{ghai2020exponentiated}. 

\section{Preliminary}
The focus of this paper is \textbf{OCO} with the decision variable taken from a compact convex subset $\cK \subseteq \mathbb{X}$ of finite dimensional vector space equipped with a norm $\norm{\cdot}$. Given a sequence of vectors $\{v_t\}$, we use the compressed-sum notation $v_{1:t}= \sum_{s=1}^tv_s$ for simplicity. We denote by $\mathbb{X}_*$ the dual space with the dual norm $\norm{\cdot}_*$. The bi-linear map combining vectors in $\mathbb{X}_*$ and $\mathbb{X}$ is denoted by 
\[
\inner{\cdot}{\cdot}:\mathbb{X}_*\times \mathbb{X}\to\mathbb{R}, (\theta, x)\mapsto \theta x.
\]
For $\mathbb{X}=\Rd$, we denote by $\norm{\cdot}_1$ the $\ell_1$ norm, the dual norm of which is the maximum norm denoted by $\norm{\cdot}_\infty$. It is well known that the $\ell_2$ norm denoted by $\norm{\cdot}_2$ is self-dual. In case $\mathbb{X}$ is the space of the matrices, for simplicity, we also use $\norm{\cdot}_1$, $\norm{\cdot}_2$ and $\norm{\cdot}_\infty$ for the nuclear, Frobenius and spectral norm, respectively. 

Let $\sigma:\mathbb{R}^{m,n}\to \mathbb{R}^{\min\{m,n\}}$ be the function mapping a matrix to its singular values. Define 
\[
\diag:\mathbb{R}^{\min\{m,n\}}\to \mathbb{R}^{m,n}, x\mapsto X
\]
with 
    \[
    X_{ij}=       \begin{cases}
            x_i, &\text{if }  i=j \\ 
            0, & \text{otherwise}. \\
            \end{cases}
    \]
Clearly, the singular value decomposition (\textbf{SVD}) of a matrix $X$ can be expressed as 
\[
X=U\diag(\sigma(X))V^\top.
\]
Similarly, we write the eigendecomposition of a symmetric matrix $X$ as
\[
X=U\diag(\lambda(X))U^\top,
\]
where we denote by $\lambda:\mathbb{S}^d\mapsto \Rd$ the function mapping a symmetric matrix to its spectrum.

Given a convex set $\cK \subseteq \mathbb{X}$ and a convex function $f:\cK\to\mathbb{R}$ defined on $\cK$, we denote by $\partial f(y)=\{g\in \mathbb{X}_*\lvert\forall y\in \cK .f(x)-f(y)\geq\inner{g}{x-y}\}$ the subgradient of $f$ at $y$. We refer to $\grad f(y)$ any element in $\partial f(y)$. A function is $\eta$-strongly convex with respect to $\norm{\cdot}$ over $\cK$ if
\[
f(x)-f(y)\geq \inner{\grad f(y)}{x-y}+\frac{\eta}{2}\norm{x-y}^2 
\]
holds for all $x,y\in \cK$ and $\grad f(y)\in\partial f(y)$. 

\section{Algorithms and Analysis}
In this section, we present and analyse our algorithms, which begins with a short review on \textbf{EG} and the $p$-norm algorithm for the case $f_t= l_t$. 
The \textbf{EG} algorithm can be considered as an instance of \textbf{OMD}, the update rules of which is given by
\[
x_{t+1,i}\propto\exp( \ln(x_{t,i})-\frac{1}{\eta}g_{t,i}),
\]
where $g_t\in\partial f_t(x_t)$ is the subgradient, and $\eta> 0$ is the stepsize. Although the algorithm has the expected logarithmic dependence on the dimensionality, its update rule is applicable only to the decision variables on the standard simplex. For the problem with decision variables taken from an $\ell_1$ ball $\{x\lvert\norm{x}_1\leq D\}$, one can apply the $\textbf{EG}^\pm$ trick, i.e. use the vector $[\frac{D}{2}g_t^\top,-\frac{D}{2}g_t^\top]^\top$ to update $[x_{t+1,+}^\top,x_{t+1,-}^\top]^\top$ at iteration $t$ and choose the decision variable $x_{t+1,+}-x_{t+1,-}$. However, if the decision set is implicitly given by a regularisation term, the parameter $D$ has to be tuned. Since applying an overestimated $D$ increases regret, while using an underestimated $D$ decreases the freedom of the model, the algorithm is sensitive to tuning. For composite objectives, \textbf{EG} is not practical due to its update rule.

Compared to \textbf{EG}, the $p$-norm algorithm, the update rule of which is given by
\[
\begin{split}
y_{t+1,i}=&\sgn(x_{t,i})\abs{x_{t,i}}^{p-1}\norm{x_t}_{p}^{\frac{2}{p-1}}-\frac{1}{\eta}g_{t,i}\\
x_{t+1,i}=&\sgn(y_{t+1,i})\abs{y_{t+1,i}}^{q-1}\norm{y_{t+1}}_{q}^{\frac{2}{q-1}},\\
\end{split}
\]
is better applicable for unknown $D$.  
To combine the ideas of \textbf{EG} and the $p$-norm algorithm, we consider the following generalised entropy function
\begin{equation}
\label{eq:entropy}
\scfi: \mathbb{R}\to \mathbb{R}, x\mapsto \alpha (\abs{x}+\beta)\ln(\frac{\abs{x}}{\beta}+1)-\alpha\abs{x}.
\end{equation}
In the next lemma, we show the twice differentiability and strict convexity of $\scfi$, based on which a strongly convex potential function for \textbf{OMD} in a compact decision set can be constructed.
\begin{lemma}
\label{lemma:property}
$\scfi$ is twice continuous differentiable and strictly convex with
\begin{enumerate}
    \item $\scfi'(x)=\alpha \ln (\frac{\abs{x}}{\beta}+1)\sgn(x)$
    \item $\scfi''(x)=\frac{\alpha}{\abs{x}+\beta}$.
\end{enumerate}
Furthermore, the convex conjugate given by $\scfi^*:\mathbb{R}\to\mathbb{R}, \theta\mapsto\alpha \beta \exp\frac{\abs{\theta}}{\alpha}-\beta\abs{\theta}-\alpha\beta$
is also twice continuous differentiable with
\begin{enumerate}
    \item $\scfi ^{*\prime}(\theta)=(\beta \exp\frac{\abs{\theta}}{\alpha}-\beta)\sgn(\theta)$
    \item $\scfi ^{*\prime\prime}(\theta)=\frac{\beta}{\alpha} \exp\frac{\abs{\theta}}{\alpha}.$
\end{enumerate}
\end{lemma}
Since we can expand the natural logarithm as $\ln(\frac{\abs{x}}{\beta}+1)=\frac{\abs{x}}{\beta}-\frac{\abs{x}^2}{2\beta^2}+\frac{\abs{x}^3}{3\beta^3}-\ldots$,  $\scf(x)$ can be intuitively considered as an interpolation between the absolute value and square. As observed in Figure \ref{fig:reg}, it is closer to the absolute value compared to the hyperbolic entropy introduced in \cite{ghai2020exponentiated}. 
Moreover, running \textbf{OMD} with regulariser $x\mapsto \sum_{i=1}^d\scf(x_i)$ yields an update rule 
\[
\begin{split}
y_{t+1,i}=& \sgn(x_{t,i})\ln(\frac{\abs{x_{t,i}}}{\beta}+1)-\frac{1}{\alpha}g_{t,i}\\
x_{t+1,i}=&\sgn(y_{t+1,i})(\beta\exp(\abs{y_{t+1,i}})-\beta),\\
\end{split}
\]
which sets the signs of coordinates like the $p$-norm algorithm and updates the scale similarly to \textbf{EG}. As illustrated in Figure \ref{fig:map}, the mirror map $\grad \dualscf$ is close to the mirror map of \textbf{EG}, while the behavior of \textbf{HU} is more similar to the gradient descent update.
\begin{figure}
\centering
\begin{subfigure}{.5\textwidth}
  \centering
  \includegraphics[width=\linewidth]{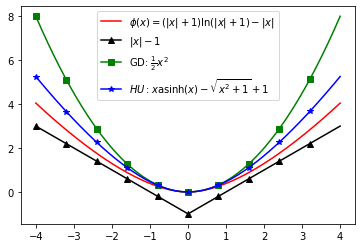}
  \caption{Regularisers}%
\label{fig:reg}
\end{subfigure}%
\begin{subfigure}{.5\textwidth}
  \centering
  \includegraphics[width=\linewidth]{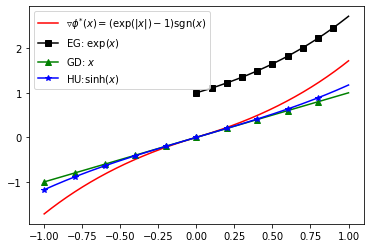}
\caption{Mirror Maps}%
\label{fig:map}
\end{subfigure}
\caption{Comparison of Convex Regularisers}
\label{fig:comp_reg}
\end{figure}

\subsection{Algorithms in the Euclidean Space}
To obtain an adaptive and optimistic algorithm, we define the following time varying function 
\begin{equation}
\label{eq:reg_t}
\sct{t}:\Rd\to \mathbb{R}, x\mapsto \alpha_t \sum_{i=1}^d((\abs{x_i}+\beta)\ln(\frac{\abs{x_i}}{\beta}+1)-\abs{x_i}),
\end{equation}
and apply it to the adaptive optimistic \textbf{OMD} (\textbf{AO-OMD}) given by
\begin{equation}
\label{eq:omd}
\begin{split}
    x_{t+1} &=\argmin_{x\in\cK}\inner{g_t-h_t+h_{t+1}}{x}+\comp_{t+1}(x)+\bd{\sct{t+1}}{x}{x_t}
\end{split}
\end{equation}
for the sequence of subgradients $\sequ{g}$ and hints $\sequ{h}$. 
In a bounded domain, $\sct{t}$ is strongly convex with respect to $\norm{\cdot}_1$, which is shown in the next lemma. 
\begin{lemma}
\label{lem:covex-rd}
Let $\cK\subseteq \Rd$ be convex and bounded such that $\norm{x}_1\leq D$ for all $x\in \cK$. Then we have for all $x,y\in \cK$
\[
\sct{t}(x)-\sct{t}(y)\geq \grad\sct{t}(y)^\top(x-y)+\frac{\alpha_t}{D+d\beta}\norm{x-y}_1^2.
\]
\end{lemma}
With the property of the strong convexity, the regret of \textbf{AO-OMD} with regulariser~\eqref{eq:reg_t} can be analysed in the framework of optimistic algorithm \citep{joulani2017modular} and is upper bounded by the following theorem.
\begin{theorem}
\label{thm:regret-omd}
Let $\cK\subseteq \Rd$ be a compact convex set. Assume that there is some $D>0$ such that $\norm{x}_1\leq D$ holds for all $x\in\cK$. Let $\sequ{x}$ be the sequence generated by update rule~\eqref{eq:omd} with regulariser~\eqref{eq:reg_t}. Setting $\beta=\frac{1}{d}$, $\eta=\sqrt{\frac{1}{\ln(D+1)+\ln d}}$, and $\alpha_t=\eta\sqrt{\sum_{s=1}^{t-1}\norm{g_s-h_s}^2_\infty}$, we obtain
\[
\begin{split}
\mathcal{R}_{1:T}\leq &\comp_1(x_1)+c(d,D)\sqrt{\sum_{t=1}^T\norm{g_t-h_t}_\infty^2} 
\end{split}
\]
for some $c(d,D)\in \mathcal{O}(D\sqrt{\ln (D+1)+\ln d})$.
\end{theorem}

\textbf{EG} can also be considered as an instance of \textbf{FTRL} with a constant stepsize. The update rule of the adaptive optimistic \textbf{FTRL} (\textbf{AO-FTRL}) is given by  
\begin{equation}
\label{eq:ftrl}
\begin{split}
    x_{t+1} &=\argmin_{x\in\cK}\inner{g_{1:t}+h_{t+1}}{x}+\comp_{1:t+1}(x)+\bd{\sct{t+1}}{x}{x_1}.
\end{split}
\end{equation} 
The regret of \textbf{AO-FTRL} is upper bounded by the following theorem.
\begin{theorem}
\label{thm:regret-ftrl}
Let $\cK\subseteq \Rd$ be a compact convex set with $d>e$. Assume that there is some $D\geq 1$ such that $\norm{x}_1\leq D$ holds for all $x\in\cK\subseteq \Rd$. Let $\sequ{x}$ be the sequence generated by updating rule~\eqref{eq:ftrl} with regulariser~\eqref{eq:reg_t} at iteration $t$. Setting $\beta=\frac{1}{d}$, $\eta=\sqrt{\frac{1}{\ln (D+1)+\ln d}}$ and $\alpha_t=\eta\sqrt{\sum_{s=1}^{t-1}\norm{g_s-h_s}_\infty^2}$, we obtain 
\[
\begin{split}
\mathcal{R}_{1:T}\leq &c(d,D)\sqrt{\sum_{t=1}^T\norm{g_t-h_t}_\infty^2}\\
\end{split}
\]
for some $c(d,D)\in \mathcal{O}(D\sqrt{\ln (D+1)+\ln d})$.
\end{theorem}

\subsection{Spectral Algorithms}
We now consider the setting in which the decision variables are matrices taken from a compact convex set $\cK\subseteq \mathbb{R}^{m,n}$. A direct attempt to solve this problem is to apply the updating rule \eqref{eq:omd} or \eqref{eq:ftrl} to the vectorised matrices. A regret bound of $\mathcal{O}(D\sqrt{T\ln(mn)})$ can be guaranteed if the $\ell_1$ norm of the vectorised matrices from $\cK$ are bounded by $D$, which is not optimal. In many applications, elements in $\cK$ are assumed to have bounded nuclear norm, for which the regulariser 
\begin{equation}
\label{eq:reg_mtx}
\msct{t}=\sct{t}\circ \sigma    
\end{equation}
can be applied. The next theorem gives the strong convexity of $\msct{t}$ with respect to $\norm{\cdot}_1$ over $\cK$, which allows us to use $\{\msct{t}\}$ as the potential functions in \textbf{OMD} and \textbf{FTRL}.
\begin{theorem}
\label{thm:convex-mat}
Let $\sigma:\mathbb{R}^{m,n}\to \Rd$ be the function mapping a matrix to its singular values. Then the function $\msct{t}=\sct{t}\circ\sigma$ is $\frac{\alpha_t}{2(D+\min\{m,n\}\beta)}$-strongly convex with respect to the nuclear norm over the nuclear ball with radius $D$. 
\end{theorem}
The proof of Theorem \ref{thm:convex-mat} follows the idea introduced in \cite{ghai2020exponentiated}. Define the operator 
\[
S:\mathbb{R}^{m,n}\to \mathbb{S}^{m+n}, X\mapsto \begin{bmatrix}
0 & X\\
X^\top & 0
\end{bmatrix}
\]
The set $\mathcal{X}=\{S(X)\lvert\in\mathbb{R}^{m,n}\}$ is a finite dimensional linear subspace of the space of symmetric matrices $\mathbb{S}^{m+n}$. Its dual space $\mathcal{X}_*$ determined by the Frobenius inner product can be represented by $\mathcal{X}$ itself. For any $S(X)\in\mathcal{X}$, the set of eigenvalues of $S(X)$ consists of the singular values and the negative singular values of $X$. Since $\scf$ is even, we have $\sum_{i=1}^d\scf(\sigma_i(X))=\sum_{i=1}^d\scf(\lambda_i(X))$ for symmetric $X$. The next lemma shows that both $\msct{t}\lvert_\mathcal{X}$ and $\dualmsct{t}\lvert_\mathcal{X}$ are twice differentiable.
\begin{lemma}
\label{lemma:D2}
Let $f:\mathbb{R}\to \mathbb{R}$ be twice continuously differentiable. Then the function given by
\[
F:\mathbb{S}^{d}\to \mathbb{R}, X\mapsto \sum_{i=1}^d f(\lambda_i(X))
\]
is twice differentiable. Furthermore, let $X\in\mathbb{S}^d$ be a symmetric matrix with eigenvalue decomposition \[X=U\diag(\lambda_1(X),\ldots,\lambda_d(X))U^\top.\]
Define the matrix of the divided difference $\Gamma(f,X)=[\gamma(f,X)_{ij}]$ with 
\[
\gamma(f,X)_{ij}=\begin{cases}
    \frac{f(\lambda_i(X))-f(\lambda_j(X))}{\lambda_i(X)-\lambda_j(X)},& \text{if }\lambda_i(X)\neq \lambda_j(X)\\
    f'(\lambda_i(X)),              & \text{otherwise}
\end{cases}
\]
Then for any $G,H\in\mathbb{S}^d$, we have
\[
D^2F(X)(G,H)=\sum_{i.j}\gamma(f',X)_{ij}\tilde{g}_{ij}\tilde{h}_{ij},\\
\]
where $\tilde{g}_{ij}$ and $\tilde{h}_{ij}$ are the elements of the $i$-th row and $j$-th column of the matrix $U^\top G U$ and $U^\top H U$, respectively. 
\end{lemma}
Lemma \ref{lemma:D2} implies the unsurprising positive semidefiniteness of $D^2F(X)$ for convex $f$.  
Furthermore, the exact expression of the second differential allows us to show the local smoothness of $\dualmsct{t}$ using the local smoothness of $\dualscf$. Together with Lemma \ref{lemma:Duality}, the locally strong convexity of $\msct{t}\lvert_\mathcal{X}$ can be proved. 
\begin{lemma}
\label{lemma:Duality}
Let $\mscf:\mathbb{X}\to \mathbb{R}$ be a closed convex function such that $\dualmscf$ is twice differentiable at some $\theta\in\mathbb{X}_*$ with positive definite $D^2\dualmscf(\theta)\in\mathcal{L}(\mathbb{X}_*,\mathcal{L}(\mathbb{X}_*,\mathbb{R}))$. Suppose that $D^2\dualmscf(\theta)(v,v)\leq \dualnorm{v}^2$ holds for all $v\in\mathbb{X}_*$. Then we have $D^2\mscf(D\dualmscf(\theta))(x,x)\geq \norm{x}^2$ for all $x\in\mathbb{X}$.
\end{lemma}
Lemma $\ref{lemma:Duality}$ can be considered as a generalised version of the local duality of smoothness and convexity proved in \cite{ghai2020exponentiated}. The required positive definiteness of $D^2\dualmsct{t}(\theta)$ is guaranteed by the exact expression of the second differential described in Lemma \ref{lemma:D2} and the fact $\scf^{*\prime\prime}(\theta)> 0$ for all $\theta \in\mathbb{R}$. Finally, using the construction of $\mathcal{X}$, the locally strong convexity of $\msct{t}\lvert_\mathcal{X}$ can be extended to $\msct{t}$. The complete proofs of Theorem \ref{thm:convex-mat} and the technical lemmata can be found in Appendix \ref{appendix:thm:convex-mat}.

With the property of the strong convexity, the regret of applying \eqref{eq:reg_mtx} to \textbf{AO-OMD} and \textbf{AO-FTRL} can be upper bounded by the following theorems. 
\begin{theorem}
\label{thm:regret-mat}
Let $\cK\subseteq \mathbb{R}^{m,n}$ be a compact convex set. Assume that there is some $D>0$ such that $\norm{x}_1\leq D$ holds for all $x\in\cK$. Let $\sequ{x}$ be the sequence generated by update rule~\eqref{eq:omd} with regulariser~\eqref{eq:reg_mtx} at iteration $t$. Setting $\beta=\frac{1}{\min\{m,n\}}$, $\eta=\sqrt{\frac{1}{\ln(D+1)+\ln \min\{m,n\}}}$, and $\alpha_t=\eta\sqrt{\sum_{s=1}^{t-1}\norm{g_s-h_s}^2_\infty}$, we obtain
\[
\begin{split}
\mathcal{R}_{1:T}\leq &\comp_1(x_1)+c(m,n,D)\sqrt{\sum_{t=1}^T\norm{g_t-h_t}_\infty^2} 
\end{split}
\]
with $c(m,n,D)\in\mathcal{O}(D\sqrt{\ln(D+1)+\ln \min\{m,n\}})$.
\end{theorem}

\begin{theorem}
\label{thm:regret-ftrl-mat}
Let $\cK\subseteq \mathbb{R}^{\min\{m,n\}}$ be a compact convex set with $\min\{m,n\}>e$. Assume that there is some $D\geq 1$ such that $\norm{x}_1\leq D$ holds for all $x\in\cK$. Let $\sequ{x}$ be the sequence generated by updating rule~\eqref{eq:ftrl} with time varying regulariser~\eqref{eq:reg_mtx}. Setting $\beta=\frac{1}{\min\{m,n\}}$, $\eta=\sqrt{\frac{1}{\ln(D+1)+\ln \min\{m,n\}}}$ and $\alpha_t=\eta\sqrt{\sum_{s=1}^{t-1}\norm{g_s-h_s}_\infty^2}$, we obtain
\[
\begin{split}
\mathcal{R}_{1:T}\leq &c(m,n,D)\sqrt{\sum_{t=1}^T\norm{g_t-h_t}_\infty^2},\\
\end{split}
\]
with $c(m,n,D)\in\mathcal{O}(D\sqrt{\ln(D+1)+\ln \min\{m,n\}})$.
\end{theorem}
With regulariser~\eqref{eq:reg_mtx}, both \textbf{AO-OMD} and \textbf{AO-FTRL} guarantee a regret upper bound proportional to $\sqrt{\ln \min\{m,n\}}$, which is the best known dependence on the size of the matrices.

\section{Derived Algorithms}
Given $z_{t+1}\in\mathbb{X}_*$ and a time varying closed convex function $R_{t+1}:\cK\to \mathbb{R}$, we consider the following updating rule
\begin{equation}
\label{eq:omd:impl}
\begin{split}
    y_{t+1} &=\grad \dualsct{t+1}(z_{t+1})\\
    x_{t+1} &=\argmin_{x\in\cK}R_{t+1}(x)+\bd{\sct{t+1}}{x}{y_{t+1}}.
\end{split}
\end{equation}
It is easy to verify that \eqref{eq:omd:impl} is equivalent to 
\[
\begin{split}
x_{t+1}=&\argmin_{x\in\cK}R_{t+1}(x)+\bd{\sct{t+1}}{x}{y_{t+1}}\\
=&\argmin_{x\in\cK}R_{t+1}(x)+\sct{t+1}(x)-\inner{\grad\sct{t+1}(y_{t+1})}{x}\\
=&\argmin_{x\in\cK}R_{t+1}(x)+\sct{t+1}(x)-\inner{z_{t+1}}{x}\\
\end{split}
\]
Setting $z_{t+1}=\grad\sct{t+1}(x_t)-g_t+h_t-h_{t+1}$ and $R_{t+1}=\comp_{t+1}$, we obtain the \textbf{AO-OMD} update
\[
\begin{split}
x_{t+1}=&\argmin_{x\in \cK}\inner{g_t-h_t+h_{t+1}}{x}-\inner{\grad\sct{t+1}(x_t)}{x} +\sct{t+1}(x)+r_{t+1}(x)\\
=&\argmin_{x\in \cK}\inner{g_t-h_t+h_{t+1}}{x}+r_{t+1}(x)+\bd{\sct{t+1}}{x}{x_t}.\\
\end{split}
\]
Setting $z_{t+1}=-\grad \sct{t+1}(x_1)+g_{1:t}+h_{t+1}$ and $R_{t+1}=\comp_{1:t+1}$, we obtain the \textbf{AO-FTRL} update
\[
\begin{split}
x_{t+1}=&\argmin_{x\in \cK}\inner{g_{1:t}-\theta_1+h_{t+1}}{x}+\sct{t+1}(x)+r_{1:t+1}(x).\\
\end{split}
\]
The rest of this section focuses on solving the second line of \eqref{eq:omd:impl} for some popular choices of $\comp$ and $\cK$.

\subsection{Elastic Net Regularisation}
We first consider the setting of $\cK=\Rd$ and $R_{t+1}(x)=\gamma_1 \norm{x}_1+\frac{\gamma_2}{2}\norm{x}^2_2$, which has countless applications in machine learning. It is easy to verify that the Bregman divergence associated with $\psi_{t+1}$ is given by
\[
\begin{split}
\bd{\sct{t+1}}{x}{y}=&\alpha_{t+1}\sum_{i=1}^d((\abs{x_i}+\beta)\ln(\frac{\abs{x_i}}{\beta}+1)-\abs{x_i}\\
&-(\sgn(y_i)x_i+\beta)\ln(\frac{\abs{y_i}}{\beta}+1)+\abs{y_i}).
\end{split}
\]
The minimiser of 
\[
R_{t+1}(x)+\bd{\sct{t+1}}{x}{y_{t+1}}
\]
in $\Rd$ can be simply obtained by setting the subgradient to $0$. For $\ln(\frac{\abs{y_{i,t+1}}}{\beta}+1)\leq\frac{\gamma_1}{\alpha_{t+1}}$, we set $x_{i,t+1}=0$. Otherwise, the $0$ subgradient implies $\sgn(x_{i,t+1})=\sgn(y_{i,t+1})$ and $\abs{x_{i,t+1}}$ given by the root of
\[
\begin{split}
\ln(\frac{\abs{y_{i,t+1}}}{\beta}+1)=\ln(\frac{\abs{x_{i,t+1}}}{\beta}+1)+\frac{\gamma_1}{\alpha_{t+1}}+\frac{\gamma_2}{\alpha_{t+1}}\abs{x_{i,t+1}}
\end{split}
\]
for $i=1,\ldots, d$. 
For simplicity, we set $a=\beta$, $b=\frac{\gamma_2}{\alpha_{t+1}}$ and $c=\frac{\gamma_1}{\alpha_{t+1}}-\ln(\frac{\abs{y_{i,t+1}}}{\beta}+1)$. It can be verified that $\abs{x_{i,t+1}}$ is given by
\begin{equation}
\abs{x_{i,t+1}}=\frac{1}{b}W_0(ab\exp(ab-c))-a,
\end{equation}
where $W_0$ is the principal branch of the \textit{Lambert function} and can be well approximated.
For $\gamma_2=0$, i.e. the $\ell_1$ regularised problem, $\abs{x_{i,t+1}}$ has the closed form solution
\begin{equation}
\abs{x_{i,t+1}}=\beta\exp(\ln(\frac{\abs{y_{i,t+1}}}{\beta}+1)-\frac{\gamma_1}{\alpha_{t+1}})-\beta.
\end{equation}
The implementation is described in Algorithm~\ref{alg:reg}.
\begin{algorithm}
	\caption{Solving $\min _{x\in \Rd} R_{t+1}(x)+\bd{\sct{t+1}}{x}{y_{t+1}}$}
    \label{alg:reg}
	\begin{algorithmic}
	\For{$i=1,\ldots,d$}
	    \If{$\ln(\frac{\abs{y_{i,t+1}}}{\beta}+1)\leq\frac{\gamma_1}{\alpha_{t+1}}$}
	    \State $x_{t+1,i}\gets 0$
        \Else
        \State $a\gets\beta$
        \State $b\gets\frac{\gamma_2}{\alpha_{t+1}}$
        \State $c\gets\frac{\gamma_1}{\alpha_{t+1}}-\ln(\frac{\abs{{y}_{t+1,i}}}{\beta}+1)$
        \State $x_{t+1,i}\gets\frac{1}{b}W_0(ab\exp(ab-c))-a$
        \EndIf
    \EndFor
    \State Return $x_{t+1}$
    \end{algorithmic}
\end{algorithm}
\subsection{Nuclear and Frobenius Regularisation}
Similarly, we consider $\cK=\mathbb{R}^{m,n}$ with a regulariser $R_{t+1}(x)=\gamma_1 \norm{x}_1+\frac{\gamma_2}{2}\norm{x}^2_2$ mixed with the nuclear and Frobenius norm.
The second line of update rule \eqref{eq:omd:impl} can be implemented as follows
\begin{equation}
    \label{eq:nu_reg}
    \begin{split}
        \text{Compute SVD: }y_{t+1}=&U_{t+1}\diag(\tilde{y}_{t+1})V_{t+1}^\top\\
        \text{Apply Algorithm~\ref{alg:reg}: }\tilde{x}_{t+1}=&\argmin_{x\in \Rd} R_{t+1}(x)+\bd{\sct{t+1}}{x}{\tilde{y}_{t+1}}\\
        \text{Construct: }x_{t+1}=&U_{t+1}\diag(\tilde{x}_{t+1})V_{t+1}^\top.
    \end{split}
\end{equation}
Let $y_{t+1}$ and $\tilde{y}_{t+1}$ be as defined in \eqref{eq:nu_reg}. It is easy to verify
\begin{equation}
\label{eq:mtx:opt}
\begin{split}
&\argmin _{x\in\mathbb{R}^{m,n}} R_{t+1}(x)+\bd{\msct{t+1}}{x}{y_{t+1}}\\
=&\argmin _{x\in\mathbb{R}^{m,n}} R_{t+1}(x)+\msct{t+1}(x)-\inner{U_{t+1}\diag(\grad \sct{t+1}(\tilde{y}_{t+1}))V_{t+1}^\top}{x}_F.
\end{split}
\end{equation}
From the characterisation of subgradient, it follows
\[
\begin{split}
\grad R_{t+1}(x)=U \diag(\gamma_1 \sgn(\sigma(x))+\gamma_2 \sigma(x)) V^\top,
\end{split}
\]
and 
\[
\begin{split}
\grad\msct{t}(x)=U \diag(\grad\sct{t}(\sigma(x))) V^\top,
\end{split}
\]
where $x=U\diag(\sigma(x))V^\top$ is \textbf{SVD} of $x$. Similar to the case in $\Rd$, $\tilde{x}_{t+1}$ is the root of 
\[
\gamma_1 \sgn(\sigma(x))+\gamma_2 \sigma(x)+\grad\sct{t}(\sigma(x))=\grad\sct{t}(\tilde{y}_{t+1}).
\]
The subgradient of the objective~\eqref{eq:mtx:opt} at $x_{t+1}=U_{t+1}\diag(\tilde{x}_{t+1})V_{t+1}^\top$ is clearly $0$. 

\subsection{Projection onto the Cross-Polytope}
Next, we consider the setting where $\comp_t$ is the zero function and $\cK$ is the $\ell_1$ ball with radius $D$. Clearly, we simply set $x_{t+1}=y_{t+1}$ for $\norm{y_{t+1}}_1\leq D$. Otherwise, Algorithm~\ref{ALG:Proj} describes a sorting based procedure projecting $y_{t+1}$ onto the $\ell_1$ ball with time complexity $\mathcal{O}(d\log d)$.
The correctness of the algorithm is shown in the next lemma.

\begin{algorithm}
	\caption{$\operatorname{project}(y,D,\beta)$}
	\label{ALG:Proj}
\begin{algorithmic}
        \State Sort $\abs{y_i}$ to get the permutation $p$ such that $\abs{y_{p(i)}}\leq \abs{y_{p(i+1)}}$
        \State Define $\theta(j)=\abs{y_{p(j)}}(D+(d-j+1)\beta)+\beta D-\beta\sum_{i\geq j}\abs{y_{p(i)}}$
        \State $\rho\gets\min\{j\lvert\theta(j)>0\}$
        \State $z\gets\frac{\sum_{i=\rho}^d(\abs{y_{p(i)}}+\beta)}{D+(d-\rho+1)\beta}$
        \State $x^*_i\gets\max\{\frac{\abs{y_i}+\beta}{z}-\beta,0\}\sgn(y_i)$ for $i=1\ldots d$
    \State    Return $x^*$\;
\end{algorithmic}
\end{algorithm}

\begin{lemma} 
\label{lemma:proj}
Let $y\in \Rd$ with $\norm{y}_1>D$ and $x^*$ as returned by Algorithm~\ref{ALG:Proj}, then we have 
\[
 x^*\in\argmin_{x\in \cK}\bd{\psi_{t+1}}{x}{y}.
\]
\end{lemma}
For the case that $\cK\subseteq\mathbb{R}^{m,n}$ is the nuclear ball with radius $D$ and $\norm{y_{t+1}}_1> D$, we need to solve the problem
\[
\min_{x\in\cK} \msct{t+1}(x)-\inner{U_{t+1}\diag(\grad \sct{t+1}(\tilde{y}_{t+1}))V_{t+1}^\top}{x}_F,
\]
where the constant part of the Bregman divergence is removed. From the von Neumann's trace inequality, the Frobenius inner product is upper bounded by
\[
\inner{U_{t+1}\grad \sct{t+1}(\tilde{y}_{t+1})V_{t+1}^\top}{x}_F\leq \sigma(x)^\top\grad \sct{t+1}(\tilde{y}_{t+1}).
\]
The equality holds when $x$ and $U_{t+1}\grad \sct{t+1}(\tilde{y}_{t+1})V_{t+1}^\top$ share a simultaneous \textbf{SVD}, i.e. the minimiser has an \textbf{SVD} of the form
\[
x=U_{t+1}\diag(\grad \sigma(x))V_{t+1}^\top.
\]
Thus the problem is reduced to
\[
\begin{split}
\min_{x\in\mathbb{R}^{\min\{m,n\}}}& \sct{t+1}(x)-\grad \sct{t+1}(\tilde{y}_{t+1})^\top x\\
\textrm{s.t.} \quad & \sum_{i=1}^{\min\{m,n\}} x_i\leq D\\
  &x_i\geq 0 \textrm{ for all } i=1,\ldots,\min\{m,n\},    \\
\end{split}
\]
which can be solved by Algorithm~\ref{ALG:Proj}. Thus, the projection of update rule~\eqref{eq:omd:impl} can be implemented as follows
\begin{equation}
    \begin{split}
        \text{Compute SVD: }y_{t+1}=&U_{t+1}\diag(\tilde{y}_{t+1})V_{t+1}^\top\\
        \text{Apply Algorithm~\ref{ALG:Proj}: }\tilde{x}_{t+1}=&\operatorname{project}(\tilde{y}_{t+1},D,\beta)\\
        \text{Construct: }x_{t+1}=&U_{t+1}\diag(\tilde{x}_{t+1})V_{t+1}^\top.
    \end{split}
\end{equation}
\subsection{Stochastic Acceleration}
Finally, we consider the stochastic optimisation problem of the form
\[
\min_{x\in\cK} l(x)+\comp(x),
\]
where $l:\mathbb{X}\to\mathbb{R}$ and $\comp:\cK\to\mathbb{R}_{\geq 0}$ are closed convex functions. In the stochastic setting, instead of having a direct access to $\grad l$, we query a stochastic gradient $g_t$ of $l$ at $z_t$ in each iteration $t$ with $\ex{g_t\lvert z_t}\in\partial l(z_t)$. Algorithms with a regret bound of the form $\mathcal{O}(\sqrt{\sum_{t=1}^T\dualnorm{g_t-h_t}^2})$ can be easily converted into a stochastic optimisation algorithm by applying the update rule to the scaled stochastic gradient $a_tg_t$ and hint $a_{t+1}g_t$, which is described in Algorithm~\ref{ALG:acc}.
\begin{algorithm}
\caption{Stochastic Acceleration}\label{ALG:acc}
\begin{algorithmic}
	\State Input: optimistic algorithm $\mathcal{A}$, compact convex set $\cK$ and closed convex function $\comp$\;
	   \For {$t=1,\ldots,T$}
	    \State $a_t\gets t$
	    \State $x_t$ from $\mathcal{A}$
	    \State $z_t\gets\frac{a_t}{a_{1:t}}x_t+(1-\frac{a_t}{a_{1:t}})z_{t-1}$ 
	    \State Query $g_t$ such that $\ex{g_t\lvert z_t}\in\partial l(z_t)$
	    \State Update $\mathcal{A}$ with $\cK$, $\alpha_{t+1}\comp$, scaled subgradient $a_t g_t$ and hint $a_{t+1} g_{t}$
        \EndFor
       \State Return $x_{t+1}$
\end{algorithmic}
\end{algorithm}
\cite{joulani2020simpler} has shown the convergence of accelerating \textbf{Adagrad} for the problem in $\Rd$. We extend the result to any finite dimensional normed vector space in the following corollary.
\begin{corollary}
\label{cor:acc}
Let $(\mathbb{X},\norm{\cdot})$ be a finite dimensional normed vector space and $\cK\subseteq \mathbb{X}$ a compact convex set. Denote by $\mathcal{A}$ be some optimistic algorithm generating $x_t\in\cK$  at iteration $t$. Denote by \[
\nu_t^2=\ex{\dualnorm{g_t-\grad l_t(z_t)}^2\lvert z_t}
\] the variance. If $\mathcal{A}$ has a regret upper bound in the form of 
\[
c_1+c_2\sqrt{\sum_{t=1}^T\dualnorm{a_t(g_t-g_{t-1})}^2}
\]
then there is some $L>0$ such that the error incurred by Algorithm \ref{ALG:acc} is upper bounded by
\[
	\begin{split}
	\ex{f(z_{T})-f(x)}\leq&\frac{c_1+c_2\sqrt{8\sum_{t=1}^Ta_t^2(\nu_t^2+L^2)}}{a_{1:T}}.\\
	\end{split}
\]
Furthermore, if $l$ is $M$-smooth, then we have
\[
	\begin{split}
	\ex{f(z_{T})-f(x)}\leq&\frac{ c_1+c_2\sqrt{8\sum_{t=1}^Ta_t^2\nu_t^2}+\sqrt{2}c_2L+2Mc_2^2}{a_{1:T}}.\\
	\end{split}
\]
\end{corollary}
Setting $\alpha_t=t$, we obtain a convergence of $\mathcal{O}(\frac{c_2}{\sqrt{T}})$ in general case, and $\mathcal{O}(\frac{c_2}{T^2}+\frac{c_2\max_{t}\nu_t}{\sqrt{T}})$ for smooth loss function. Applying update rule \eqref{eq:omd} or \eqref{eq:ftrl} with regulariser \eqref{eq:reg_t} or \eqref{eq:reg_mtx} to Algorithm \ref{ALG:acc}, the constant $c_2$ is proportional to $\sqrt{\ln d}$ and $\sqrt{\ln(\min\{m,n\})}$ for $\mathbb{X}=\Rd$ and $\mathbb{X}=\mathbb{R}^{m,n}$ respectively, while the accelerated \textbf{AdaGrad} has a linear dependence on the dimensionality \citep{joulani2020simpler}. 

\section{Experiments}
This section shows the empirical evaluation of the developed algorithms. We carry out experiments on both synthetic and real-world data and demonstrate the performances of the \textbf{OMD} (\textbf{Exp-MD}) and \textbf{FTRL} (\textbf{Exp-FTRL}) based on the exponentiated update. 

\subsection{Online Logistic Regression}
\label{sec:logit}
For a sanity check, we simulate an $d$-dimensional online logistic regression problem, in which the model parameter $w^*$ has a $99\%$ sparsity and the non-zero values are randomly drawn from the uniform distribution over $[-1,1]$. At each iteration $t$, we sample a random feature vector $x_t$ from a uniform distribution over $[-1,1]^d$ and generate a label $y_t\in \{-1,1\}$ using a logit model, i.e. $\operatorname{Pr}[y_t=1]=(1+\exp(-w^\top x_t))^{-1}$. The goal is to minimise the cumulative regret 
\[
\mathcal{R}_{1:T}=\sum_{t=1}^Tl_t(w_t)-\sum_{t=1}^Tl_t(w^*)
\]
with $l_t(w)=\ln(1+\exp(-y_tw^\top x_t))$. We choose $d=10,000$ and compare our algorithms with \textbf{AdaGrad}, \textbf{AdaFTRL} \citep{duchi2011adaptive} and \textbf{HU} \citep{ghai2020exponentiated}. For both \textbf{AdaGrad} and \textbf{AdaFTRL}, we set the $i$-th entry of the proximal matrix $H_t$ to $h_{ii}=10^{-6}+\sum_{s=1}^{t-1}g_{s,i}^2$ as their theory suggested \citep{duchi2011adaptive}. The stepsize of \textbf{HU} is set to $\sqrt{\frac{1}{\sum_{s=1}^{t-1}\norm{g_{s}}_\infty^2}}$ leading to an adaptive regret upper bound. All algorithms take decision variables from an $\ell_1$ ball $\{w\in\Rd\lvert \norm{w}_1\leq D\}$, which is the ideal case for \textbf{HU}. We examine the performances of the algorithms with known, underestimated and overestimated $\norm{w^*}_1$ by setting $D=\norm{w^*}_1$, $D=\frac{1}{2}\norm{w^*}_1$ and $D=2\norm{w^*}_1$, respectively. For each choice of $D$, we simulate the online process of each algorithm for $10,000$ iterations and repeat the experiments for $20$ trials. 

Figure~\ref{fig:logistic} plots the curves of the average cumulative regret with the ranges of standard deviation as shaded regions. As can be observed, our algorithms have a clear and stable advantage over the \textbf{AdaGrad}-style algorithms and slightly outperform \textbf{HU} in the experiments with known $\norm{w^*}_1$. As the combination of the entropy-like regulariser and \textbf{FTRL} can also be used for parameter-free optimisation \citep{cutkosky2017online}, overestimating $\norm{w^*}_1$ does not have a tangible impact on the performance of \textbf{Exp-FTRL}, which leads to its clear advantage over the rest.
\begin{figure}
\centering
\begin{subfigure}{.45\textwidth}
  \centering
  \includegraphics[width=\linewidth]{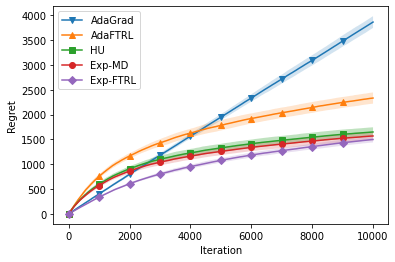}
  \caption{D=$\norm{w^*}_1$}%
\label{fig:logistic-known}
\end{subfigure}%

\begin{subfigure}{.45\textwidth}
  \centering
  \includegraphics[width=\linewidth]{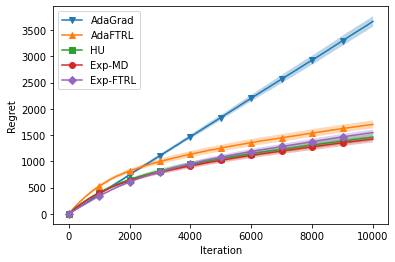}
\caption{$D=\frac{1}{2}\norm{w^*}_1$}%
\label{fig:logistic-under}
\end{subfigure}
\begin{subfigure}{.45\textwidth}
  \centering
  \includegraphics[width=\linewidth]{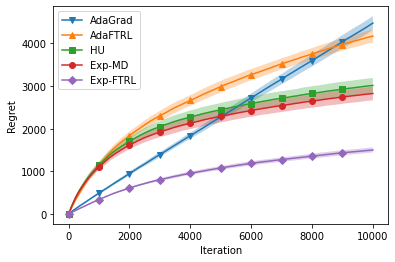}
\caption{$D=2\norm{w^*}_1$}%
\label{fig:logistic-over}
\end{subfigure}
\caption{Online Logistic Regression}
\label{fig:logistic}
\end{figure}

\subsection{Online Multitask Learning}
Next, we examine the performance of the developed spectral algorithms using a simulated online multi-task learning problem \citep{kakade2012regularization}, in which we need to solve $k$ highly correlated $d$-dimensional online prediction problems simultaneously. The data are generated as follows. We first randomly draw two orthogonal matrices $U\in\operatorname{GL}(d,\mathbb{R})$ and $V\in\operatorname{GL}(k,\mathbb{R})$. Then we generate a $k$-dimensional vector $\sigma$ with $r$ non-zero values randomly drawn from a uniform distribution over $[0,10]$ and construct a low rank parameter matrix $W^*=U\diag(\sigma)V$. At each iteration $t$, $k$ feature and label pairs $(x_{t,1},y_{t,1}),\ldots,(x_{t,k},y_{t,k})$ are generated using $k$ logit models with the $i$-th parameters taken from the $i$-th rows of $W$. The loss function is given by $l_t(W)=\sum_{i=1}^k\ln(1+\exp(-y_{t,i}w_i^\top x_{t,i}))$. We set $d=100$, $k=25$ and $r=5$, take the nuclear ball $\{W\in\mathbb{R}^{d,k}\lvert \norm{W}_1\leq D\}$ as the decision set and run the experiment as in subsection~\ref{sec:logit}. The average and standard deviation of the results over $20$ trials are shown in Figure~\ref{fig:multi-task}. 

Similar to the online logistic regression, our algorithms have a clear advantage over \textbf{AdaGrad} and \textbf{AdaFTRL} and slightly outperform \textbf{HU} in all settings. While the regret of the \textbf{AdaGrad}-style algorithms spread over a wider range, our algorithms yield relatively stabler results. The superiority of \textbf{Exp-FTRL} for the overestimated $\norm{W^*}_1$ can also be observed from figure~\ref{fig:multi-task-over}.
\begin{figure}
\centering
\begin{subfigure}{.45\textwidth}
  \centering
  \includegraphics[width=\linewidth]{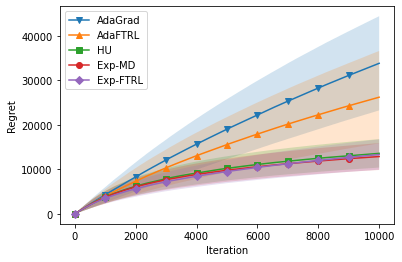}
  \caption{$D=\norm{W^*}_1$}%
\label{fig:multi-task-known}
\end{subfigure}%

\begin{subfigure}{.45\textwidth}
  \centering
  \includegraphics[width=\linewidth]{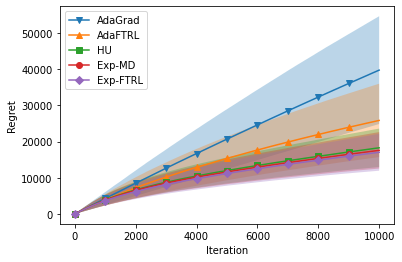}
\caption{Underestimated $D=\frac{1}{2}\norm{W^*}_1$}%
\label{fig:multi-task-under}
\end{subfigure}
\begin{subfigure}{.45\textwidth}
  \centering
  \includegraphics[width=\linewidth]{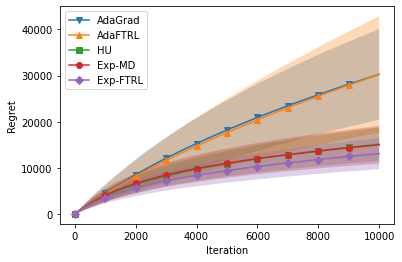}
\caption{Overestimated $D=2\norm{W^*}_1$}%
\label{fig:multi-task-over}
\end{subfigure}
\caption{Online Multitask Learning}
\label{fig:multi-task}
\end{figure}

\subsection{Optimisation for Contrastive Explanations}
Generating the contrastive explanation of a machine learning model \citep{NEURIPS2018_c5ff2543} is the most motivating application of this paper. Given a sample $x_0\in \mathcal{X}$ and machine learning model $f:\mathcal{X}\to \mathbb{R}^K $, the contrastive explanation consists of a set of pertinent positive (\textbf{PP}) features and a set of pertinent negative (\textbf{PN}) features, which can be found by solving the following optimisation problem \citep{NEURIPS2018_c5ff2543} 
\[
\begin{split}
\min_{x\in\mathcal{W}}\quad &l_{x_0}(x)+\lambda_1\norm{x}_1+\frac{\lambda_2}{2}\norm{x}_2^2.\\
\end{split}
\]
Let $\kappa\geq 0$ be a constant and define $k_0=\arg\max_{i}f(x_0)_i$.
The loss function for finding \textbf{PP} is given by
\[
l_{x_0}(x)=\max\{\max_{i\neq k_0}f(x)_i-f(x)_{k_0},-\kappa\},
\]
which imposes a penalty on the features that do not justify the prediction. \textbf{PN} is the set of features altering the final classification and is modelled by the following loss function
\[
l_{x_0}(x)=\max\{f(x_0+x)_{k_0}-\max_{i\neq k_0}f(x_0+x)_{i},-\kappa\}.
\]

In the experiment, we first train a ResNet$20$ model \citep{7780459} on the CIFAR-$10$ dataset \citep{krizhevsky2009learning}, which attains a test accuracy of $91.49\%$. For each class of the images, we randomly pick $100$ correctly classified images from the test dataset and generate \textbf{PP} and \textbf{PN} for them. For \textbf{PP}, we take the set of all feasible images as the decision set, while for \textbf{PN}, we take the set of tensors $x$, such that $x_0+x$ is a feasible image. 

We first consider the white-box setting, in which we have the access to $\grad l_{x_0}$. Our goal is to demonstrate the performance of the accelerated \textbf{AO-OMD} and \textbf{AO-FTRL} based on the exponentiated update (\textbf{AccAOExpMD} and \textbf{AccAOExpFTRL}).
In \cite{NEURIPS2018_c5ff2543}, the fast iterative shrinkage-thresholding algorithm (\textbf{FISTA}) \citep{beck2009fast} is applied to finding the \textbf{PP} and \textbf{PN}. Therefore, we take \textbf{FISTA} as our baseline. In addition, our algorithms are also compared with the accelerated \textbf{AO-OMD} and \textbf{AO-FTRL} with \textbf{AdaGrad}-style stepsizes (\textbf{AccAOMD} and \textbf{AccAOFTRL}) \citep{joulani2020simpler}. 

We pick $\lambda_1=\lambda_2=\frac{1}{2}$, which is the largest value from the set $\{2^{-i}\lvert i\in\mathbb{N}\}$ allowing \textbf{FISTA} to attain a negative loss $l_{x_0}$ for $10$ randomly selected images. All algorithms start from $x_1=0$. Figure~\ref{fig:WB} plots the convergence behaviour of the five algorithms, averaged over the $1000$ images. In the experiment for \textbf{PP}, our algorithms are obviously better than the \textbf{AdaGrad}-style algorithms. Although \textbf{FISTA} converges faster at the first $100$ iterations, it does not make further progress afterwards due to the tiny stepsize found by the backtracking rule. In the experiment for \textbf{PN}, all algorithms behave similarly. It is worth pointing out that the backtracking rule of \textbf{FISTA} requires multiple function evaluations, which are expensive for explaining deep neural networks. 
\begin{figure}
\centering
\begin{subfigure}{.5\textwidth}
  \centering
  \includegraphics[width=\linewidth]{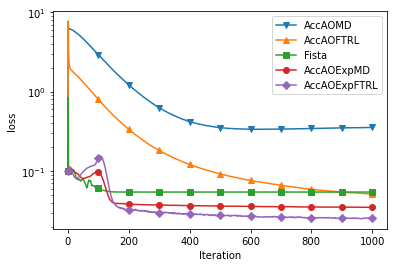}
\caption{Convergence for Generating \textbf{PP}}%
\label{fig:PP}
\end{subfigure}%
\begin{subfigure}{.5\textwidth}
  \centering
  \includegraphics[width=\linewidth]{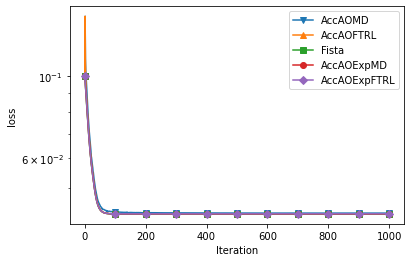}
\caption{Convergence for Generating  \textbf{PN}}%
\label{fig:PN}
\end{subfigure}
\caption{White Box Contrastive Explanations on CIFAR-$10$ }c
\label{fig:WB}
\end{figure}

Next, we consider the black-box setting, in which the gradient is estimated through the two-points estimation
\[
\frac{1}{b}\sum_{i=1}^{b}\frac{\delta}{\mu}(f(x+\mu v_i)-f(x))v_i,
\]
where $\delta$, $\mu$ are constants and $v_i$ is a random vector. Following \cite{chen2019zo}, we set $\delta=d$ and sample $v_i$ independently from the uniform distribution over the unit sphere for \textbf{AdaGrad}-style algorithms. Since the convergence of our algorithms depends on the variance of the gradient estimation in $(\Rd, \norm{\cdot}_\infty)$, we set $\delta=1$ and sample $\nu_{i,1},\ldots,\nu_{i,d}$ independently from Rademacher distribution according to Corollary~3 in \cite{duchi2015optimal}. To ensure a small bias of the gradient estimation, we set $\mu=\frac{1}{\sqrt{dT}}$, which is the recommended value for non-convex and constrained optimisation in \cite{chen2019zo}. The performances of the algorithms are examined in the high and low variance settings with $b=1$ and $b=\sqrt{T}$, respectively. Since the problem is stochastic, \textbf{FISTA}, which searches for the stepsize at each iteration, is not practical. Thus, we remove it from the comparison. 

Figure~\ref{fig:BB} plots the convergence behaviour of the algorithms in the high variance setting. Our algorithms outperform the \textbf{AdaGrad}-style algorithms for generating both \textbf{PP} and \textbf{PN}. Furthermore, the \textbf{FTRL} based algorithms have higher convergence rates than the \textbf{MD} based ones at the first few iterations, leading to overall better performance. The experimental results of the low variance setting are plotted in figure~\ref{fig:BB-batch}. Though \textbf{AccAOExpFTRL} yields the smallest objective value at the beginning of the experiments, it gets stuck in the local minimum around $0$ and is outperformed by \textbf{AccAOExpMD} and \textbf{AccAOFTRL} at the later iterations. Overall, the algorithms based on the exponentiated update have an advantage over the \textbf{AdaGrad}-style algorithms for both high and low variance settings. 
\begin{figure}
\centering
\begin{subfigure}{.5\textwidth}
  \centering
  \includegraphics[width=\linewidth]{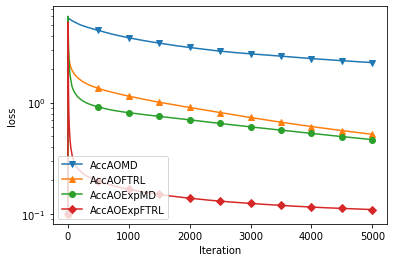}
\caption{Convergence for Generating \textbf{PP}}%
\label{fig:bb-pp}
\end{subfigure}%
\begin{subfigure}{.5\textwidth}
  \centering
  \includegraphics[width=\linewidth]{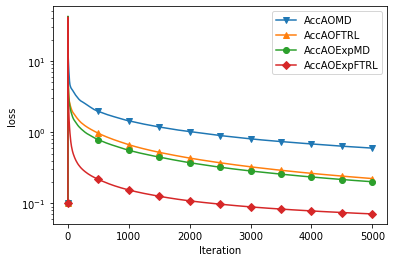}
\caption{Convergence for Generating  \textbf{PN}}%
\label{fig:bb-pn}
\end{subfigure}
\caption{Black Box Contrastive Explanations: High Variance Setting}
\label{fig:BB}
\end{figure}
\begin{figure}
\centering
\begin{subfigure}{.5\textwidth}
  \centering
  \includegraphics[width=\linewidth]{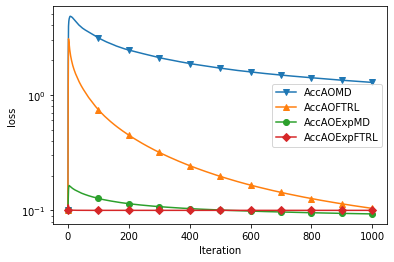}
\caption{Convergence for Generating \textbf{PP}}%
\label{fig:bb-pp-batch}
\end{subfigure}%
\begin{subfigure}{.5\textwidth}
  \centering
  \includegraphics[width=\linewidth]{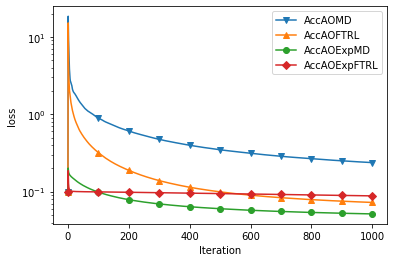}
\caption{Convergence for Generating  \textbf{PN}}%
\label{fig:bb-pn-batch}
\end{subfigure}
\caption{Black Box Contrastive Explanations: Low Variance Setting}
\label{fig:BB-batch}
\end{figure}

\section{Conclusion}
This paper proposes and analyses a family of online optimisation algorithms based on an entropy-like regulariser combined with the ideas of optimism and adaptivity. The proposed algorithms have adaptive regret bounds depending logarithmically on the dimensionality of the problem, can handle popular composite objectives and can be easily converted into stochastic optimisation algorithms with optimal accelerated convergence rates for smooth function. As a future research direction, we plan to analyse the convergence of the proposed algorithms together with variance reduction techniques for non-convex stochastic optimisation and analyse their empirical performance for training deep neural networks.  

\backmatter

\section*{Declarations}
\subsection*{Funding}
The research leading to these results received funding from the German Federal Ministry for Economic Affairs and Climate Action under Grant Agreement No. 01MK20002C.
\subsection*{Code availability}
The implementation of the experiments and all algorithms involved in the experiments are available on GitHub \url{https://github.com/mrdexteritas/exp_grad}.
\subsection*{Availability of data and materials}
The source code generating synthetic data, creating neural networks and model training are available on GitHub \url{https://github.com/mrdexteritas/exp_grad}. The CIFAR-10 data are collected from \url{https://www.cs.toronto.edu/~kriz/cifar.html}.
\subsection*{Conflicts of Interests and Competing Interests}
The authors declare that they have no conflicts of interests or competing interests.
\subsection*{Ethics Approval}
Not Applicable.
\subsection*{Consent to Participate}
Not Applicable
\subsection*{Consent for Publication}
Not Applicable
\subsection*{Authors' Contributions}
Conceptualization: WS; Methodology: WS; Formal analysis and investigation: WS; Software: WS; Validation: WS, FS; Visualization: WS; Writing - original draft preparation: WS; Writing - review and editing: WS, FS; Funding acquisition: SA; Resources: SA; Supervision: FS, SA.
\noindent
\bigskip





\begin{appendices}

\section{Missing Proofs of Section 3.1}
\subsection{Proof of Lemma \ref{lemma:property}}
\begin{proof}
It is straightforward that $\scfi$ is differentiable at $x\neq 0$ with \[\scfi'(x)=\alpha \ln (\frac{\abs{x}}{\beta}+1)\sgn(x).\] For any $h\in\mathbb{R}$, we have
\[
\begin{split}
\scfi(0+h)-\scfi(0)=&\alpha (\abs{h}+\beta)\ln(\frac{\abs{h}}{\beta}+1)-\alpha\abs{h}\\
\leq & \alpha (\abs{h}+\beta)\frac{\abs{h}}{\beta}-\alpha\abs{h}\\
=& \frac{\alpha}{\beta}h^2,
\end{split}
\]
where the first inequality uses the fact $\ln x\leq x-1$.
Further more, we have 
\[
\begin{split}
\scfi(0+h)-\scfi(0)=&\alpha (\abs{h}+\beta)\ln(\frac{\abs{h}}{\beta}+1)-\alpha\abs{h}\\
\geq & \alpha (\abs{h}+\beta)(\frac{\abs{h}}{\abs{h}+\beta})-\alpha\abs{h}\\
\geq & 0,
\end{split}
\]
where the first inequality uses the fact $\ln x\geq 1-\frac{1}{x}$.
Thus, we have 
\[
0\leq \frac{\scfi(0+h)-\scfi(0)}{h}\leq \frac{\alpha}{\beta}h
\]
for $h>0$ and 
\[
 \frac{\alpha}{\beta}h\leq \frac{\scfi(0+h)-\scfi(0)}{h}\leq 0.
\]
for $h<0$, from which it follows $\lim_{h\to 0} \frac{\scfi(0+h)-\scfi(0)}{h}=0$.
Similarly, we have for $x\neq 0$ \[
\scfi''(x)=\frac{\alpha}{\abs{x}+\beta}.
\]
Let $h\neq 0$. Then we have
\[
\frac{\scfi'(0+h)-\scfi'(0)}{h}=\frac{\alpha \ln(\frac{\abs{h}}{\beta}+1)\sgn(h)}{h}=\frac{\alpha \ln(\frac{\abs{h}}{\beta}+1)}{\abs{h}}.
\]
From the inequalities of the logarithm, it follows
\[
\frac{\alpha}{\abs{h}+\beta}\leq\frac{\scfi'(0+h)-\scfi'(0)}{h}\leq \frac{\alpha}{\beta}.
\]
Thus, we obtain $\scfi''(0)=\frac{\alpha}{\beta}$.
By the definition of the convex conjugate  we have
\begin{equation}
\label{eq:conj}
\scfi ^*(\theta)=\max_{x\in\mathbb{R}} \theta x-\scfi(x),
\end{equation}
which is differentiable.
The maximiser $y$ satisfies 
\[
\ln(\frac{\abs{y}}{\beta}+1)\sgn(y)=\theta.
\]
Since $\ln(\frac{\abs{y}}{\beta}+1)\geq 0$ holds, we have $\sgn(y)=\sgn(\theta)$ and 
\[
\abs{y}=\beta{\exp(\frac{\abs{\theta}}{\alpha})-\beta}.
\]
Thus, we obtain the maximiser $y=\scfi^{*\prime}(\theta)$ by setting

\begin{equation}
\label{eq:conj_p}
y=\sgn(\theta)(\beta{\exp(\frac{\abs{\theta}}{\alpha})-\beta}).
\end{equation}
Combining \eqref{eq:conj} and \eqref{eq:conj_p}, we obtain
\[
\scfi^*(\theta)=\alpha \beta \exp\frac{\abs{\theta}}{\alpha}-\beta\abs{\theta}-\alpha\beta.
\]
To prove that $\scfi^*$ is twice differentiable, it suffices to show that $\scfi^{*\prime}$ is differentiable at $0$. For any $h\neq 0$, we have
\[
\frac{\scfi^{*\prime}(0+h)-\scfi^{*\prime}(0)}{h}=\frac{\sgn(h)(\beta\exp(\frac{\abs{h}}{\alpha})-\beta)}{h}.
\]
Applying the inequalities of the logarithm, we obtain
\[
\frac{\beta}{\alpha}\leq \frac{\sgn(h)(\beta\exp(\frac{\abs{h}}{\alpha})-\beta)}{h}\leq \frac{\beta}{\alpha}\exp(\frac{\abs{h}}{\alpha}),
\]
from which it follows $\scfi^{*}$ is twice differentiable at $0$ and 
\[
\scfi^{*\prime\prime}(0)=\frac{\beta}{\alpha}.
\]
\end{proof}

\subsection{Proof of Lemma \ref{lem:covex-rd}}
\begin{proof}
Let $x\in \cK$ be arbitrary. We have
\[
\begin{split}
    v^\top \grad ^2 \sct{t}(x) v=&\alpha_t\sum_{i=1}^d \frac{v_i^2}{\abs{x_i}+\beta}\\
    =& \alpha_t\sum_{i=1}^d \frac{v_i^2}{\abs{x_i}+\beta}\sum_{i=1}^d(\abs{x_i}+\beta)\frac{1}{\sum_{i=1}^d(\abs{x_i}+\beta)}\\
    \geq &\frac{\alpha_t}{\sum_{i=1}^d(\abs{x_i}+\beta)}(\sum_{i=1}^d \abs{v_i})^2\\
    \geq &\frac{\alpha_t}{D+d\beta}(\sum_{i=1}^d \abs{v_i})^2\\
    = &\frac{\alpha_t}{D+d\beta}\norm{v}_1^2
\end{split}
\]
for all $v\in \Rd$, where the first inequality follows from the Cauchy-Schwarz inequality. This leads clearly to the strong convexity for a twice differentiable function.
\end{proof}

\subsection{Proof of Theorem \ref{thm:regret-omd}}
\begin{proposition}
\label{prop:omd}
Let $\cK\subseteq \mathbb{X}$ be a convex set. Assume that $\comp_t:\cK\to \mathbb{R}_{\geq 0}$ is closed convex function defined on $\cK$ and $\psi_{t}:\cK\mapsto \mathbb{R}$ is $\eta_t$-strongly convex w.r.t. $\norm{\cdot}$ over $\cK$. Then the sequence $\sequ{x}$ generated by \eqref{eq:omd} with regulariser $\sequ{\psi}$ guarantees 
\[
\begin{split}
\mathcal{R}_{1:T}\leq&\comp_1(x_1)+\bd{\sct{1}}{x}{x_1}+\sum_{t=1}^{T}(\bd{\sct{t+1}}{x}{x_t}-\bd{\sct{t}}{x}{x_t})+\sum_{t=1}^{T}\frac{\dualnorm{g_t-h_t}^2}{2\eta_{t+1}}.\\
\end{split}
\]
\begin{proof}
From the optimality condition, it follows that for all $x\in \cK$
\[
\begin{split}
&\inner{g_t-h_t+h_{t+1}+\grad \comp_{t+1}(x_{t+1})}{x_{t+1}-x} \\
\leq &\inner{\grad \sct{t+1}(x_{t})-\grad\sct{t+1}(x_{t+1})}{x-x_{t+1}}\\
=&\bd{\sct{t+1}}{x}{x_t}-\bd{\sct{t+1}}{x}{x_{t+1}}-\bd{\sct{t+1}}{x_{t+1}}{x_t}.
\end{split}
\]
Then, we have
\[
\begin{split}
&\inner{g_t}{x_{t}-x}+\comp_{t+1}(x_{t+1})-\comp_{t+1}(x)\\
\leq &\inner{g_t}{x_{t}-x_{t+1}}+\inner{g_t-h_t+h_{t+1}+\grad r_{t+1}(x_{t+1})}{x_{t+1}-x}\\
&+\inner{h_t-h_{t+1}}{x_{t+1}-x}\\
\leq &\inner{g_t-h_t}{x_{t}-x_{t+1}}+\inner{h_t}{x_t-x}-\inner{h_{t+1}}{x_{t+1}-x}\\
&+\bd{\sct{t+1}}{x}{x_t}-\bd{\sct{t+1}}{x}{x_{t+1}}-\bd{\sct{t+1}}{x_{t+1}}{x_t}\\
\end{split}
\]
Adding up from $1$ to $T$ , we obtain
\[
\begin{split}
&\sum_{t=1}^T(\inner{g_t}{x_{t}-x}+\comp_{t+1}(x_{t+1})-\comp_{t+1}(x))\\
\leq &\sum_{t=1}^T\inner{g_t-h_t}{x_{t}-x_{t+1}}+\sum_{t=1}^T(\inner{h_t}{x_t-x}-\inner{h_{t+1}}{x_{t+1}-x})\\
&+\sum_{t=1}^T(\bd{\sct{t+1}}{x}{x_t}-\bd{\sct{t+1}}{x}{x_{t+1}}-\bd{\sct{t+1}}{x_{t+1}}{x_t})\\
\leq &\sum_{t=1}^T(\inner{g_t-h_t}{x_{t}-x_{t+1}}-\bd{\sct{t+1}}{x_{t+1}}{x_t})\\
&+\inner{h_1}{x_1-x}-\inner{h_{T+1}}{x_{T+1}-x}\\
&+\bd{\sct{1}}{x}{x_1}+\sum_{t=1}^{T}(\bd{\sct{t+1}}{x}{x_t}-\bd{\sct{t}}{x}{x_t})\\
\end{split}
\]
$h_1$, $h_{T+1}$ and $x_{T+1}$, which are artifacts of the analysis, can be set to $0$. Then, we simply obtain
\[
\begin{split}
&\sum_{t=1}^T(\inner{g_t}{x_t-x}+\comp_t(x_{t})-\comp_t(x))\\
=&\sum_{t=1}^T(\inner{g_t}{x_t-x}+\comp_{t+1}(x_{t+1})-\comp_{t+1}(x))\\
&+\comp_1(x_1)-\comp_1(x)-\comp_{T+1}(x_{T+1})+\comp_{T+1}(x)\\
\leq&\sum_{t=1}^T(\inner{g_t}{x_t-x}+\comp_{t+1}(x_{t+1})-\comp_{t+1}(x))+\comp_1(x_1)-\comp_1(x)+\comp_{T+1}(x)\\
\leq&\comp_1(x_1)-\comp_1(x)+\comp_{T+1}(x)+\sum_{t=1}^T(\inner{g_t-h_t}{x_{t}-x_{t+1}}-\bd{\sct{t+1}}{x_{t+1}}{x_t})\\
&+\bd{\sct{1}}{x}{x_1}+\sum_{t=1}^{T}(\bd{\sct{t+1}}{x}{x_t}-\bd{\sct{t}}{x}{x_t})\\
\end{split}
\]
Since $\comp_{T+1}$ is not involved in the regret, we assume without loss of generality $\comp_1=\comp_{T+1}$. From the $\eta_t$-strong convexity of $\sct{t}$ we have 
\[
\begin{split}
&\inner{g_t-h_t}{x_{t}-x_{t+1}}-\bd{\sct{t+1}}{x_{t+1}}{x_t}\\
\leq &\inner{g_t-h_t}{x_{t}-x_{t+1}}-\frac{\eta_{t+1}}{2}\norm{x_t-x_{t+1}}^2\\
\leq & \dualnorm{g_t-h_t}\norm{x_t-x_{t+1}}-\frac{\eta_{t+1}}{2}\norm{x_t-x_{t+1}}^2\\
\leq & \frac{\dualnorm{g_t-h_t}^2}{2\eta_{t+1}}+\frac{\eta_{t+1}}{2}\norm{x_t-x_{t+1}}^2-\frac{\eta_{t+1}}{2}\norm{x_t-x_{t+1}}^2\\
=&\frac{\dualnorm{g_t-h_t}^2}{2\eta_{t+1}},
\end{split}
\]
where the second inequality uses the definition of dual norm, the third inequality follows from the fact $ab\leq \frac{a^2}{2}+\frac{b^2}{2}$.
The claimed the result follows.
\end{proof}
\end{proposition}

\begin{proof}[Proof of Theorem \ref{thm:regret-omd}]
Proposition \ref{prop:omd} can be directly applied, and we obtain
\begin{equation}
\label{eq1:thm1}
\begin{split}
\mathcal{R}_{1:T}\leq &\sum_{t=1}^{T}(\bd{\sct{t+1}}{x}{x_{t}}-\bd{\sct{t}}{x}{x_{t}})+\sum_{t=1}^T\frac{D+d\beta}{2\alpha_t}\norm{g_t-h_t}_\infty^2\\
&+\bd{\sct{1}}{x}{x_1}+r(x_1).\\
\end{split}
\end{equation}
Using Lemma \ref{lemma:bd_upper}, we bound the first term of \eqref{eq1:thm1}
\[
\begin{split}
&\sum_{t=1}^{T}(\bd{\sct{t+1}}{x}{x_{t}}-\bd{\sct{t}}{x}{x_{t}})\\
\leq & 4D(\ln(D+1)+\ln d)\sum_{t=2}^{T}(\alpha_{t+1}-\alpha_t)\\
\leq & 4D(\ln(D+1)+\ln d)\alpha_{T+1}\\
\leq & 4D(\ln(D+1)+\ln d)\eta \sqrt{\sum_{t=1}^T\norm{g_t-h_t}^2_\infty}.
\end{split}
\]
Using Lemma \ref{lemma:log}, the second term of \eqref{eq1:thm1} can be bounded as
\[
\begin{split}
\sum_{t=1}^T\frac{(D+1)\norm{g_t-h_t}_\infty^2}{4\alpha_t}\leq &\frac{D+1}{2\eta}\sqrt{\sum_{t=1}^T\norm{g_t-h_t}_\infty^2}\\
\end{split}
\]
The third term of \eqref{eq1:thm1} is simply $0$ since we set $\alpha_1=0$. Setting $\eta=\sqrt{\frac{1}{\ln(D+1)+\ln d}}$ and combining the inequalities above, we obtain the claimed result.
\end{proof}

\subsection{Proof of Theorem \ref{thm:regret-ftrl}}
\begin{proposition}
\label{prop:ftrl}
Let $\cK\subseteq \mathbb{X}$ be a compact convex set such that $\norm{x}\leq D$ holds for all $x\in \cK$, $\comp_t:\cK\to \mathbb{R}_{\geq 0}$ and $\sct{t}:\cK\mapsto \mathbb{R}$ closed convex function defined on $\cK$. Assume  $\sct{t}$ is $\eta_t$-strongly convex w.r.t. $\norm{\cdot}$ over $\cK$ and $\sct{t}\leq \sct{t+1}$ for all $t=1,\ldots,T$. Then the sequence $\sequ{x}$ generated by \eqref{eq:ftrl} with guarantees 
\begin{equation}
\label{ftrl:regret:eq}
\begin{split}
\mathcal{R}_{1:T}\leq \sct{T+1}(x)+\sum_{t=1}^T\frac{2D\dualnorm{g_t-h_t}^2}{\sqrt{16D^2\eta_t
^2+\dualnorm{g_t-h_t}^2}}.
\end{split}
\end{equation}
\begin{proof}[Proof of Proposition \ref{prop:ftrl}]
First, define $\psi_t=\comp_{1:t}+\sct{t}$. Then, we have 
\[
\begin{split}
&\sum_{t=1}^T\psi_{t+1}^*(\theta_{t+1}-h_{t+1})-\psi^*_{t}(\theta_t-h_t)\\
=&\psi^*_{T+1}(\theta_{T+1}-h_{T+1})-\psi^*_{1}(\theta_{1}-h_1)\\
\geq&\inner{\theta_{T+1}-h_{T+1}}{x}-\psi_{T+1}(x)-\psi^*_{1}(\theta_{1}-h_1)\\
\geq&\inner{-\sum_{t=1}^Tg_t-h_{T+1}}{x}-\psi_{T+1}(x)-\psi^*_{1}(\theta_{1}-h_1)\\
\end{split}
\]
Setting the artifacts $h_{T+1}$ to $0$, rearranging and adding $\sum_{t=1}^{T}\inner{g_t}{w_t}$ to both sides, we obtain 
\[
    \begin{split}
	&\sum_{t=1}^T\inner{g_t}{x_t-x}\\
	\leq&\psi_{T+1}(x)+\psi^*_{1}(\theta_{1}-h_1)+\sum_{t=1}^T(\psi^*_{t+1}(\theta_{t+1}-h_{t+1})-\psi^*_{t}(\theta_t-h_{t})+\inner{g_t}{x_t})\\
	=&\psi_{T+1}(x)-\inner{h_1}{x_1}-r_1(x_1)\\
	&+\sum_{t=1}^T(\psi^*_{t+1}(\theta_{t+1}-h_{t+1})-\psi^*_{t}(\theta_{t+1}))\\
	&+\sum_{t=1}^T(\psi^*_{t}(\theta_{t+1})-\psi_t^*(\theta_t-h_{t})+\inner{\theta_{t}-\theta_{t+1}}{\grad\psi_{t}^*(\theta_t-h_t)})\\
	\leq&\psi_{T+1}(x)-\inner{h_1}{x_1}-r_1(x_1)\\
	&+\sum_{t=1}^T(\psi^*_{t+1}(\theta_{t+1}-h_{t+1})-\psi^*_{t}(\theta_{t+1}))\\
	&+\sum_{t=1}^T(\psi^*_{t}(\theta_{t+1})-\psi_t^*(\theta_t-h_{t})+\inner{\theta_{t}-\theta_{t+1}}{\grad\psi_{t}^*(\theta_t-h_t)}),\\
    \end{split}
\]
From the definition of $\psi_t$, it follows 
\[
\psi_{T+1}(x)=\sct{T+1}(x)+r_{1:T+1}(x)=\sct{T+1}(x)+r_{1:T}(x),
\] 
where we assumed $r_{T+1}\equiv 0$, since it is not involved in the regret.
Furthermore, we have for $t\geq 1$
\[
\begin{split}
&\psi^*_{t+1}(\theta_{t+1}-h_{t+1})-\psi^*_{t}(\theta_{t+1})\\
\leq&\inner{\theta_{t+1}-h_{t+1}}{x_{t+1}}-\psi_{t+1}(x_{t+1})-\inner{\theta_{t+1}}{x_{t+1}}+\psi_{t}(x_{t+1})\\
=&-\inner{h_{t+1}}{x_{t+1}}-\psi_{t+1}(x_{t+1})+\psi_{t}(x_{t+1})\\
=&-\inner{h_{t+1}}{x_{t+1}}-\comp_{1:t+1}(x_{t+1})+\comp_{1:t}(x_{t+1})-\sct{t+1}(x_{t+1})+\sct{t}(x_{t+1})\\
\leq&-\inner{h_{t+1}}{x_{t+1}}-\comp_{t+1}(x_{t+1}),\\
\end{split}
\]
where the first inequality uses the definition of convex conjugate and the second inequality follows from the fact $\sct{t+1}\leq \sct{t}$.
Adding up from $1$ to $T$, we obtain
\[
\begin{split}
&\sum_{t=1}^T(\psi^*_{t+1}(\theta_{t+1}-h_{t+1})-\psi^*_{t}(\theta_{t+1}))\\
\leq&-\sum_{t=1}^T\comp_{t+1}(x_{t+1})-\sum_{t=1}^T\inner{h_{t+1}}{x_{t+1}}\\
=&\comp_1(x_1)+\inner{h_{1}}{x_1}-\comp_{T+1}(x_{t+1})-\inner{h_{T+1}}{x_{T+1}}-\sum_{t=1}^T\comp_t(x_{t})-\sum_{t=1}^T\inner{h_{t}}{x_{t}}\\
=&\comp_1(x_1)+\inner{h_{1}}{x_1}-\sum_{t=1}^T\comp_t(x_{t})-\sum_{t=1}^T\inner{h_{t}}{x_{t}},\\
\end{split}
\]
where we use $r_{T+1}\equiv 0$ and $h_{T+1}=0$.
Combining the inequality above and rearranging, we have
\begin{equation}
\label{prop:ftrl:eq1}
\begin{split}
&\sum_{t=1}^T(\inner{g_t}{x_t-x}+\comp_t(x_t)-\comp_t(x))\\
\leq &\sct{T+1}(x)+\sum_{t=1}^T(\psi^*_{t}(\theta_{t+1})-\psi_t^*(\theta_t-h_{t})+\inner{\theta_{t}-h_t-\theta_{t+1}}{\grad\psi_{t}^*(\theta_t-h_t)})\\
\leq &\sct{T+1}(x)+\sum_{t=1}^T\bd{\psi_t^*}{\theta_{t+1}}{\theta_t-h_t}.\\
\end{split}
\end{equation}
Next, by the definition of the Bregman divergence, we have
\[
\begin{split}
    &\bd{\psi_t^*}{\theta_{t+1}}{\theta_t-h_t}\\
    \leq&\inner{\theta_{t+1}}{\grad\psi_t^*(\theta_{t+1})}-\psi_{t}(\grad\psi_t^*(\theta_{t+1}))-\inner{\theta_t-h_t}{x_t}+\psi_t(x_t)+\inner{g_t-h_t}{x_t}\\
    =&\inner{\theta_{t}-h_t}{\grad\psi_t^*(\theta_{t+1})-x_t}-\psi_{t}(\grad\psi_t^*(\theta_{t+1}))+\psi_t(x_t)+\inner{g_t-h_t}{x_t-\grad\psi_t^*(\theta_{t+1})}\\
    =&\inner{g_t-h_t}{x_t-\grad\psi_t^*(\theta_{t+1})}-\bd{\psi_t}{\grad\psi_t^*(\theta_{t+1})}{x_t}.\\
\end{split}
\]
Since $\sct{t}$ is $\eta_{t}$ strongly convex, we have
\begin{equation}
\label{eq:prop:ftrl:1}
\begin{split}
&\inner{g_t-h_t}{x_t-\grad\psi_t^*(\theta_{t+1})}-\bd{\psi_{t}}{\grad\psi_t^*(\theta_{t+1})}{x_t}\\
\leq &\frac{1}{2\eta_t}\dualnorm{g_t-h_t}^2+\frac{\eta_t}{2}\norm{x_t-\grad\psi_t^*(\theta_{t+1})}^2-\bd{\psi_{t}}{\grad\psi_t^*(\theta_{t+1})}{x_t}\\
\leq &\frac{1}{2\eta_t}\dualnorm{g_t-h_t}^2\\
\end{split}
\end{equation}
We also have 
\begin{equation}
\label{eq:prop:ftrl:2}
\begin{split}
&\inner{g_t-h_t}{x_t-\grad\psi_t^*(\theta_{t+1})}-\bd{\psi_{t}}{\grad\psi_t^*(\theta_{t+1})}{x_t}\\
\leq &\inner{g_t-h_t}{x_t-\grad\psi_t^*(\theta_{t+1})}\\
\leq &2D\dualnorm{g_t-h_t}.\\
\end{split}
\end{equation}
Putting \eqref{eq:prop:ftrl:1} and \eqref{eq:prop:ftrl:2} together, we have
\[
\begin{split}
&\inner{g_t-h_t}{x_t-\grad\psi_t^*(\theta_{t+1})}-\bd{\psi_{t}}{\grad\psi_t^*(\theta_{t+1})}{x_t}\\
\leq &\min\{\frac{1}{2\eta_t}\dualnorm{g_t-h_t}^2,2D\dualnorm{g_t-h_t}\}\\
\leq &\frac{1}{\frac{2\eta_t}{\dualnorm{g_t-h_t}^2}+\frac{1}{2D\dualnorm{g_t-h_t}}}\\
\leq &\frac{2D\dualnorm{g_t-h_t}^2}{4D\eta_t+\dualnorm{g_t-h_t}}\\
\leq &\frac{2D\dualnorm{g_t-h_t}^2}{\sqrt{16D^2\eta_t
^2+\dualnorm{g_t-h_t}^2}}\\
\end{split}
\]
Combining the inequalities above, we obtain
\[
\mathcal{R}_{1:T}\leq \sct{T+1}(x)+\sum_{t=1}^T\frac{2D\dualnorm{g_t-h_t}^2}{\sqrt{16D^2\eta_t
^2+\dualnorm{g_t-h_t}^2}}
\]
\end{proof}
\end{proposition}
\begin{proof}[Proof of Theorem \ref{thm:regret-ftrl}]
We take the Bregman divergence $\bd{\sct{t}}{x}{x_1}$ as the regulariser at iteration $t$. 
Since $\bd{\sct{t}}{x}{x_1}$ is non-negative, increasing with $t$ and $\frac{2\alpha_t}{D+\beta d}$ strongly-convex w.r.t. $\norm{\cdot}_1$, Proposition \ref{prop:ftrl} can be directly applied, and we get
\[
\begin{split}
\mathcal{R}_{1:T}\leq&\bd{\sct{T+1}}{x}{x_1}+\sum_{t=1}^T\frac{2D\norm{g_t-h_t}_\infty^2}{\sqrt{\frac{64D^2\alpha_t^2}{(D+\beta d)^2}
+\norm{g_t-h_t}_\infty^2}}\\
=&\bd{\sct{T+1}}{x}{x_1}+\frac{2D}{\eta}\sum_{t=1}^T\frac{\norm{g_t-h_t}_\infty^2}{\sqrt{\frac{64D^2}{(D+\beta d)^2}\sum_{s=1}^{t-1}\norm{g_s-h_t}^2_\infty
+\frac{1}{\eta^2}\norm{g_t-h_t}_\infty^2}}\\
\end{split}
\]
Setting $\beta=\frac{1}{d}$ and $\eta=\frac{1}{\sqrt{\ln(D+1)+\ln d}}$, we have
\[
\begin{split}
&\frac{\norm{g_t-h_t}_\infty^2}{\sqrt{\frac{64D^2}{(D+\beta d)^2}\sum_{s=1}^{t-1}\norm{g_s-h_t}^2_\infty
+\frac{1}{\eta^2}\norm{g_t-h_t}_\infty^2}}\\
=&\frac{\norm{g_t-h_t}_\infty^2}{\sqrt{\frac{64D^2}{(D+1)^2}\sum_{s=1}^{t-1}\norm{g_s-h_t}^2_\infty
+(\ln(D+1)+\ln d)\norm{g_t-h_t}_\infty^2}}\\
\leq &\frac{\norm{g_t-h_t}_\infty^2}{\sqrt{\sum_{s=1}^{t-1}\norm{g_s-h_t}^2_\infty
+\norm{g_t-h_t}_\infty^2}}\\
=&\frac{\norm{g_t-h_t}_\infty^2}{\sqrt{\sum_{s=1}^{t}\norm{g_s-h_t}^2_\infty}},\\
\end{split}
\]
where the inequality uses the assumptions $D\geq 1$ and $d>e$. Adding up from $1$ to $T$, we obtain
\[
\begin{split}
\mathcal{R}_{1:T}\leq &\bd{\sct{T+1}}{x}{x_1}+2D\sqrt{\ln(D+1)+\ln d}\sum_{t=1}^T\frac{\norm{g_t-h_t}_\infty^2}{\sqrt{\sum_{s=1}^{t}\norm{g_s-h_t}^2_\infty}}\\
\leq & \bd{\sct{T+1}}{x}{x_1}+4D\sqrt{\ln(D+1)+\ln d}\sqrt{\sum_{t=1}^T\norm{g_t-h_t}_\infty^2}
\end{split}
\]
The first term can be bounded by Lemma \ref{lemma:bd_upper}
\[
\begin{split}
\bd{\sct{T+1}}{x}{x_1}\leq &4D\sqrt{\ln(D+1)+\ln d}\sqrt{\sum_{t=1}^T\norm{g_t-h_t}_\infty^2}\\
\end{split}
\]
Combining the inequality above, we obtain
\[
\begin{split}
\mathcal{R}_{1:T}\leq &c(D,d)\sqrt{\sum_{t=1}^T\norm{g_t-h_t}_\infty^2},\\
\end{split}
\]
with $c(D,d)\in \mathcal{O}(D\sqrt{\ln(D+1)+\ln d})$, which is the claimed result.
\end{proof}

\section{Missing Proofs of Section 3.2}
\subsection{Proof of Theorem \ref{thm:convex-mat}}
\label{appendix:thm:convex-mat}
The Proof of Theorem \ref{thm:convex-mat} is based on the idea of \cite{ghai2020exponentiated}. We first revise some technical lemmata.
\begin{proof}[Proof of Lemma \ref{lemma:D2}]
Define $\Tilde{F}:\mathbb{S}^d\to \mathbb{S}^d, X\mapsto U\diag(f(\lambda_1(X)),\ldots,f(\lambda_d(X)))U^\top$. Apparently, we have $F(X)=\tr\tilde{F}(X)$.
From the Theorem V.3.3 in \cite{bhatia2013matrix}, it follows that $\tilde{F}$ is differentiable and
\[
D\Tilde{F}(X)(H)=U({\Gamma}(f,X)\odot U^\top HU)U^\top.
\] 
Using the linearity of the trace and the chain rule, $F$ is differentiable and the directional derivative at $X$ in $H$ is given by
\[
\begin{split}
D_HF(X)=&D\tr(\tilde{F}(X))\circ D\tilde{F}(X)(H)\\
=&\tr(D\tilde{F}(X)(H))\\
=&\tr(U(\tilde{\Gamma}(f,X)\odot U^\top HU)U^\top)\\
=&\tr(\tilde{\Gamma}(f,X)\odot U^\top HU)\\
=&\sum_{i=1}^df'(\lambda_i(X))\tilde{h}_{ii}\\
=&\tr( U\diag(f'(\lambda_1(X)),\ldots, f'(\lambda_d(X)))U^\top H)
\end{split}
\]
where $\tilde{h}_{ii}$ is the $i$-th element in the diagonal of the matrix $U^\top HU$. Next, define 
\[
\Bar{F}:\mathbb{S}^d\to \mathbb{S}^d, X\mapsto U\diag(f'(\lambda_1(X)),\ldots, f'(\lambda_d(X)))U^\top.
\]
And we have 
\[
DF(X)=H\mapsto\tr(\bar{F}(X)H)
\]
Applying Theorem V.3.3 in \cite{bhatia2013matrix} again, we obtain the differentiability of $\bar{F}$ and 
\[
\begin{split}
D\Bar{F}(X)(G)=U(\Gamma(f',X)\odot U^\top GU)U^\top.
\end{split}
\]
Note that $X\mapsto \tr(X(\cdot))$ is a linear map between finite dimensional spaces. Thus $F$ is twice differentiable. From the linearity of the trace operator and matrix multiplication, it follows that $D_HF(X)$ is differentiable. Applying the chain rule, we obtain 
\[
\begin{split}
D^2F(X)(G,H)=&D_G(D_HF)(X)\\
=&D(D_H F)(X)(G)\\
=&\tr((D\Bar{F}(X)(G))H)\\
=&\tr(U(\Gamma(f',X)\odot U^\top GU)U^\top H)\\
=&\tr((\Gamma(f',X)\odot U^\top GU) U^\top HU)\\
=&\sum_{i.j}\gamma(f',X)_{ij}\tilde{g}_{ij}\tilde{h}_{ij},\\
\end{split}
\]
which is the claimed result.
\end{proof}

\begin{proof}[Proof of Lemma \ref{lemma:Duality}]
Since $D^2\dualmscf(\theta)\in\mathcal{L}(\mathbb{X}_*,\mathcal{L}(\mathbb{X}_*,\mathbb{R}))$ is positive definite and $\mathbb{X}$ is finite dimensional, the map
\[
f_\theta:\mathbb{X}_*\to \mathbb{X}, v\mapsto D^2\dualmscf(\theta)(v,\cdot)
\]
is invertible. Furthermore, defining $\psi_\theta:\mathbb{X}_*\to \mathbb{R}, v\mapsto \frac{1}{2} D^2\dualmscf(\theta)(v,v)$, we have
\[
\begin{split}
D\psi_\theta(v)=&\frac{1}{2}D^2\dualmscf(\theta)(v,\cdot)+\frac{1}{2}D^2\dualmscf(\theta)(\cdot,v)\\
=&f_\theta(v).
\end{split}
\]
Thus, we obtain the convex conjugate $\psi_\theta^*$
\[
\begin{split}
    \psi_\theta^*(x)=&\sup_{v\in\mathbb{X}_*}\inner{v}{x}-\psi_\theta(v)\\
    =&\inner{f^{-1}_\theta(x)}{x}-\psi_{\theta}(f^{-1}_\theta(x))\\
    =&\inner{f^{-1}_\theta(x)}{x}-\frac{1}{2}\inner{f^{-1}_\theta(x)}{D^2\dualmscf(\theta)(f^{-1}_\theta(x),\cdot)}\\
    =&\inner{f^{-1}_\theta(x)}{x}-\frac{1}{2}\inner{f^{-1}_\theta(x)}{f_\theta(f^{-1}_\theta(x))}\\
    =&\frac{1}{2}\inner{f^{-1}_\theta(x)}{x}\\
\end{split}
\]
by setting $x=D\psi_\theta(v)$. Denote by $I:\mathbb{X}\to \mathbb{X}, x\mapsto x$ the identity function. From $D\dualmscf=D\mscf^{-1}$, it follows
\[
\begin{split}
    I(x)=&D I(v)(x)\\
    =&D(D\dualmscf\circ D\mscf)(v)(x)\\
    =&D^2\dualmscf(D\mscf(v))\circ D^2\mscf(v)(x),\\
    =&D^2\dualmscf(\theta)\circ D^2\mscf(D\dualmscf(\theta))(x)\\
\end{split}
\]
for $\theta=D\mscf(v)$ and all $x\in\mathbb{X}$. Thus, we have $f_\theta^{-1}= D^2\mscf(D\dualmscf(\theta))$ and 
\[ 
\begin{split}
\psi_\theta^*(x)=&\frac{1}{2}\inner{f^{-1}_\theta(x)}{x}\\
=&\frac{1}{2}D^2\mscf(D\dualmscf(\theta))(x,x).
\end{split}
\]
Finally, since $\psi_\theta(v)\leq \frac{1}{2}\dualnorm{v}^2$ holds for all $v\in\mathbb{X}_*$, we can reverse the order by applying Proposition 2.19 in \cite{barbu2012convexity} and obtain for all $x\in\mathbb{X}$
\[
\frac{1}{2}D^2\mscf(D\dualmscf(\theta))(x,x)=\psi_\theta^*(x)\geq \frac{1}{2}\norm{x}^2,
\]
which is the claimed result.
\end{proof}
Finally, we can prove Theorem \ref{thm:convex-mat}.
\begin{proof}[Proof of Theorem \ref{thm:convex-mat}]
We start the proof by introducing the required definitions. Define the operator 
\[
S:\mathbb{R}^{m,n}\to \mathbb{S}^{m+n}, X\mapsto \begin{bmatrix}
0 & X\\
X^\top & 0
\end{bmatrix}
\]
The set $\mathcal{X}=\{S(X)\lvert X\in\mathbb{R}^{m,n}\}$ is a finite dimensional linear subspace of the space of symmetric matrices $\mathbb{S}^{m+n}$, and thus $(\mathcal{X},\norm{\cdot}_1)$ is a finite dimensional Banach space. Its dual space $\mathcal{X}_*$ determined by the Frobenius inner product can be represented by $\mathcal{X}$ itself. Denote by $\mathbb{B}(D)=\{X\in\mathbb{R}^{m,n}\lvert \norm{X}_1\leq D\}$ the nuclear ball with radius $D$. Then the set $\cK=\{S(X)\lvert X\in \mathbb{B}(D)\}$ is a nuclear ball in $\mathcal{X}$ with radius $2D$, since $\norm{S(X)}_1=2\norm{X}_1$ for all $X\in\mathbb{R}^{m,n}$.

Let $S(X)\in \cK$ be arbitrary. Denote by $F_t=\msct{t}\vert_{\mathcal{X}}$ the restriction of $\msct{t}$ to $\mathcal{X}$. Next, we show the strong convexity of $F_t$ over $\cK$. From the conjugacy formula of Theorem 2.4 in \cite{lewis1995convex} and Lemma \ref{lemma:property}, it follows
\[
\begin{split}
F_t^*(S(X))=&\sct{t}^*\circ\sigma(S(X))=\sct{t}^*\circ\lambda(S(X)),\\
\end{split}
\]
where the second equality follows from the fact that $\dualmsct{t}$ is absolutely symmetric.
By Lemma \ref{lemma:property} and Lemma \ref{lemma:D2}, $F_t^*$ is twice differentiable. Let $X\in\cK$ be arbitrary and $\Theta=DF_t(X)\in\mathcal{X}_*$. For simplicity, we define
\[
f_t:\mathbb{R}\to\mathbb{R}, x\mapsto \alpha_t \beta \exp\frac{\abs{x}}{\alpha_t}-\beta\abs{x}-\alpha_t\beta.
\] 
Then, for all $H\in\mathcal{X}$,
\[
D^2F_t^*(\Theta)(H,H)=\sum_{ij}\gamma(f_t',\Theta)_{ij}\tilde{h}_{ij}^2,
\]
where $\Gamma(f_t',\Theta)=[\gamma(f_t',\Theta)_{ij}]$ is the matrix of the second divided difference with 
\[
\gamma(f_t',\Theta)_{ij}=\begin{cases}
    \frac{f_t'(\lambda_i(\Theta))-f_t'(\lambda_j(\Theta))}{\lambda_i(\Theta)-\lambda_j(\Theta)},& \text{if }\lambda_i(\Theta)\neq \lambda_j(\Theta)\\
    f_t^{\prime\prime}(\lambda_i(\Theta)),              & \text{otherwise.}
\end{cases}
\]
$D^2F_t^*(\Theta)$ is clearly positive definite over $\mathbb{S}^{m+n}$, since $\gamma(f_t',\Theta)_{ij}>0$ for all $i$ and $j$. Furthermore, from the mean value theorem and the convexity of $f_t^{\prime\prime}$, there is a $c_{ij}\in (0,1)$ such that 
\[
\begin{split}
\frac{f_t^{\prime}(\lambda_i(\Theta))-f_t^{\prime}(\lambda_j(\Theta))}{\lambda_i(\Theta)-\lambda_j(\Theta)}\leq &f_t^{\prime\prime}(c_{ij}\lambda_i(\Theta)+(1-c_{ij})\lambda_j(\Theta))\\
\leq&c_{ij}f_t^{\prime\prime}(\lambda_i(\Theta))+(1-c_{ij})f_t^{\prime\prime}(\lambda_j(\Theta))\\
\leq&f_t^{\prime\prime}(\lambda_i(\Theta))+f_t^{\prime\prime}(\lambda_j(\Theta))\\
\end{split}
\]
holds for all $\lambda_i(\Theta)\neq \lambda_j(\Theta)$. Thus, we obtain
\begin{equation}
\label{eq1:theo:spec}
\begin{split}
D^2F_t^*(\Theta)(H,H)=&\sum_{ij}\gamma(f_t,\Theta)_{ij}\tilde{h}_{ij}^2\\
\leq &\sum_{ij}(f_t^{\prime\prime}(\lambda_i(\Theta))+f_t^{\prime\prime}(\lambda_j(\Theta)))\tilde{h}_{ij}^2\\
= &2\sum_{i=1}^{m+n}f_t^{\prime\prime}(\lambda_i(\Theta))\sum_{j=1}^{m+m}\tilde{h}_{ij}^2\\
=&2\tr(UHU^\top\diag(f_t^{\prime\prime}(\lambda_1(\Theta)),\ldots,f_t^{\prime\prime}(\lambda_{m+n}(\Theta))UHU^\top)\\
=&2\tr(H^2\diag(f_t^{\prime\prime}(\lambda_1(\Theta)),\ldots,f_t^{\prime\prime}(\lambda_{m+n}(\Theta)))\\
\leq &2\sum_{i=1}^{2\min\{m,n\}}\sigma_i(H^2)\sigma_i(\diag(f_t^{\prime\prime}(\lambda_1(\Theta)),\ldots,f_t^{\prime\prime}(\lambda_{m+n}(\Theta)))
\end{split}
\end{equation}
where the last line uses von Neumann’s trace inequality and the fact that the rank of $H\in\mathcal{X}$ and $\Theta$ is at most $2\min\{m,n\}$. Since $H^2$ is positive semi-definite, $\sigma_i(H^2)=\sigma_i(H)^2$ holds for all $i$. Furthermore, $f_t''(x)\geq 0$ holds for all $x\in\mathbb{R}$. Thus, the last line of \eqref{eq1:theo:spec} can be rewritten into
\begin{equation}
\label{eq2:theo:spec}
\begin{split}
D^2F_t^*(\Theta)(H,H)\leq &2\sum_{i=1}^{2\min\{m,n\}}\sigma_i(H)^2\sigma_i(\diag(f_t^{\prime\prime}(\lambda_1(\Theta)),\ldots,f_t^{\prime\prime}(\lambda_{m+n}(\Theta)))\\
    \leq &2\norm{H}_\infty^2\sum_{i=1}^{2\min\{m,n\}}\sigma_i(\diag(f_t^{\prime\prime}(\lambda_1(\Theta)),\ldots,f_t^{\prime\prime}(\lambda_{m+n}(\Theta)))\\
    \leq &2\norm{H}_\infty^2\sum_{i=1}^{2\min\{m,n\}}f_t^{\prime\prime}(\lambda_i(\Theta)).
\end{split}
\end{equation}
Recall $\Theta=DF_t(S(X))$ for $S(X)\in\cK$. Together with Lemma \ref{lemma:property}, we obtain
\[
\begin{split}
f_t^{\prime\prime}(\lambda_i(\Theta))=&\frac{\beta}{\alpha_t} \exp\frac{\abs{\lambda_i(\Theta)}}{\alpha_t}\\
=&\frac{\beta}{\alpha_t} \exp\frac{\abs{ \alpha_t \ln (\frac{\abs{\lambda_i(S(X))}}{\beta}+1)}}{\alpha_t}\\
=&\frac{\abs{\lambda_i(S(X))}+\beta}{\alpha_t}.\\
\end{split}
\]
By the construction of $\cK$, it is clear that $\sum_{i=1}^{2\min\{m,n\}}\abs{\lambda_i(S(X))}\leq 2D$. Thus, \eqref{eq2:theo:spec} can be simply further upper bounded by
\[
\begin{split}
D^2F_t^*(\Theta)(H,H)\leq& 2\norm{H}_\infty^2\sum_{i=1}^{2\min\{m,n\}}\frac{\abs{\lambda_i(S(X))}+\beta}{\alpha_t}\\
\leq& 2\norm{H}_\infty^2\frac{2D+2\min\{m,n\}\beta}{\alpha_t}
\end{split}
\]
Finally, applying Lemma \ref{lemma:Duality}, we obtain
\[
D^2F_t(S(X))(Y,Y)\geq \frac{\alpha_t}{4(D+\min\{m,n\}\beta)}\norm{Y}_1^2,
\]
which implies the $\frac{\alpha_t}{4(D+\min\{m,n\}\beta)}$-strong convexity of $F_t$ over $\cK$. 

Finally, we prove the strongly convexity of $\msct{t}$ over $B(D)\in\mathbb{R}^{m+n}$. Let $X,Y\in B(D)$ be arbitrary matrices in the nuclear ball. The following inequality can be obtained
\[
\begin{split}
    &2\msct{t}(X)-2\msct{t}(Y)\\
    =&\msct{t}(S(X))-\msct{t}(S(Y))\\
    \geq&\inner{D\msct{t}(S(Y))}{S(X)-S(Y)}_F+\frac{\alpha_t}{8(D+\min\{m,n\}\beta)}\norm{S(X)-S(Y)}_1^2\\
    =&2\inner{D\msct{t}(Y)}{X-Y}_F+\frac{\alpha_t}{2(D+\min\{m,n\}\beta)}\norm{X-Y}_1^2,\\
\end{split}
\]
which implies the $\frac{\alpha_t}{2(D+\min\{m,n\}\beta)}$-strong convexity of $\msct{t}$ as desired.
\end{proof}

\subsection{Proof of Theorem \ref{thm:regret-mat}}
\begin{proof}
The proof is almost identical to the proof of Theorem \ref{thm:regret-omd}. From the strong convexity of $\msct{t}$ shown in Theorem \ref{thm:convex-mat} and the general upper bound in Proposition \ref{prop:omd}, we obtain\begin{equation}
\label{eq1:thm3}
\begin{split}
\mathcal{R}_{1:T}\leq&\comp_1(x_1)+\bd{\sct{1}}{x}{x_1}+\sum_{t=1}^{T}(\bd{\sct{t+1}}{x}{x_t}-\bd{\sct{t}}{x}{x_t})+\sum_{t=1}^{T}\frac{\dualnorm{g_t-h_t}^2}{2\eta_{t+1}}.
\end{split}
\end{equation}
Using Lemma \ref{lemma:bd_upper}, we have
\[
\begin{split}
&\sum_{t=1}^{T}(\bd{\msct{t+1}}{x}{x_{t}}-\bd{\msct{t}}{x}{x_{t}})\\
\leq & 4D(\ln(D+1)+\ln \min\{m,n\})\sum_{t=2}^{T}(\alpha_{t+1}-\alpha_t)\\
\leq & 4D(\ln(D+1)+\ln \min\{m,n\})\alpha_{T+1}\\
=&4D(\ln(D+1)+\ln \min\{m,n\}) \eta \sqrt{\sum_{t=1}^T\norm{g_t-h_t}^2_\infty}\\
=&4D\sqrt{\ln(D+1)+\ln \min\{m,n\}} \sqrt{\sum_{t=1}^T\norm{g_t-h_t}^2_\infty}\\
\end{split}
\]
Furthermore, from Lemma \ref{lemma:log}, it follows
\[
\begin{split}
\sum_{t=1}^T\frac{(D+1)\norm{g_t-h_t}_\infty^2}{4\alpha_t}\leq &\frac{D+1}{2}\sqrt{\ln(D+1)+\ln \min\{m,n\}}\sqrt{\sum_{t=1}^T\norm{g_t-h_t}_\infty^2}\\
\end{split}
\]
The claimed result is obtained by combining the inequalities above. 
\end{proof}

\subsection{Proof of Theorem \ref{thm:regret-ftrl-mat}}
\begin{proof}
Since $\bd{\msct{t}}{x}{x_1}$ is non-negative, increasing and $\frac{2\alpha_t}{D+\beta d}$ strongly-convex w.r.t. $\norm{\cdot}_1$, Proposition \ref{prop:ftrl} can be directly applied, and we get
\[
\begin{split}
\mathcal{R}_{1:T}\leq&\bd{\msct{t}}{x}{x_1}+\sum_{t=1}^T\frac{2D\norm{g_t-h_t}_\infty^2}{\sqrt{\frac{64D^2\alpha_t^2}{(D+\beta d)^2}
+\norm{g_t-h_t}_\infty^2}}\\
=&\bd{\msct{t}}{x}{x_1}+\frac{2D}{\eta}\sum_{t=1}^T\frac{\norm{g_t-h_t}_\infty^2}{\sqrt{\frac{64D^2}{(D+\beta d)^2}\sum_{s=1}^{t-1}\norm{g_s-h_t}^2_\infty
+\frac{1}{\eta^2}\norm{g_t-h_t}_\infty^2}}\\
\end{split}
\]
Setting $\beta=\frac{1}{\min\{m,n\}}$ and $\eta=\frac{1}{\sqrt{\ln(D+1)+\ln \min\{m,n\}}}$, we have
\[
\begin{split}
&\frac{\norm{g_t-h_t}_\infty^2}{\sqrt{\frac{64D^2}{(D+\beta d)^2}\sum_{s=1}^{t-1}\norm{g_s-h_t}^2_\infty
+\frac{1}{\eta^2}\norm{g_t-h_t}_\infty^2}}\\
=&\frac{\norm{g_t-h_t}_\infty^2}{\sqrt{\frac{64D^2}{(D+1)^2}\sum_{s=1}^{t-1}\norm{g_s-h_t}^2_\infty
+(\ln(D+1)+\ln d)\norm{g_t-h_t}_\infty^2}}\\
\leq &\frac{\norm{g_t-h_t}_\infty^2}{\sqrt{\sum_{s=1}^{t-1}\norm{g_s-h_t}^2_\infty
+\norm{g_t-h_t}_\infty^2}}\\
=&\frac{\norm{g_t-h_t}_\infty^2}{\sqrt{\sum_{s=1}^{t}\norm{g_s-h_t}^2_\infty}},\\
\end{split}
\]
where the inequality uses the assumptions $D\geq 1$ and $\min\{m,n\}>e$. Adding up from $1$ to $T$, we obtain
\[
\begin{split}
\mathcal{R}_{1:T}\leq &\bd{\msct{t}}{x}{x_1}+2D\sqrt{\ln(D+1)+\ln \min\{m,n\}}\sum_{t=1}^T\frac{\norm{g_t-h_t}_\infty^2}{\sqrt{\sum_{s=1}^{t}\norm{g_s-h_t}^2_\infty}}\\
\leq & \bd{\msct{t}}{x}{x_1}+4D\sqrt{\ln(D+1)+\ln \min\{m,n\}}\sqrt{\sum_{t=1}^T\norm{g_t-h_t}_\infty^2}
\end{split}
\]
The first term can be bounded by Lemma \ref{lemma:bd_upper}
\[
\begin{split}
\bd{\msct{T+1}}{x}{x_1}\leq& 4D(\ln(D+1)+\ln \min\{m,n\})\sqrt{\sum_{t=1}^T\norm{g_t-h_t}_\infty^2}\\
\end{split}
\]
Combining the inequalities above, we obtain
\[
\begin{split}
\mathcal{R}_{1:T}\leq &c(D,m,n)\sqrt{\sum_{t=1}^T\norm{g_t-h_t}_\infty^2},\\
\end{split}
\]
with $c(D,m,n)\in \mathcal{O}(D\sqrt{\ln(D+1)+\ln \min\{m,n\}})$, which is the claimed result.
\end{proof}

\section{Missing Proofs of Section 3.4}
\subsection{Proof of Lemma \ref{lemma:proj}}
\begin{proof}[Proof of Lemma \ref{lemma:proj}]
Let $x^*$ be the minimiser of $\bd{\psi_{t+1}}{x}{y_{t+1}}$ in $\cK$. Using the the fact $\ln a\geq 1-\frac{1}{a}$, we obtain 
\[
\ln(\frac{\abs{x^*_i}}{\beta}+1)\geq \frac{\abs{x^*_i}}{\abs{x^*_i+\beta}}
\] and \[
((\abs{x^*_i}+\beta)\ln(\frac{\abs{x^*_i}}{\beta}+1)-\abs{x^*_i}\geq 0.
\] 
Thus, $y_{i}=0$ implies $x^*_i=0$. Furthermore $\sgn(x^*_i)=\sgn(y_{i})$ must hold for all $i$ with $y_{i}\neq 0$, since otherwise we can always flip the sign of $x^*_i$ to obtain smaller objective value. So we assume without loss of generality that $y_{i}\geq 0$. We claim that $\sum_{i=1}^d x^*_i=D$ holds for the minimiser $x^*$. If it is not the case, there must be some $i$ with $x^*_i< y_{i}$, and increasing $x^*_i$ by a small enough amount can decrease the objective function. Thus minimising the Bregman divergence can be rewritten into
\begin{equation}
\label{min:simplex}
\begin{split}
\min_{x\in\Rd}\quad &\sum_{i=1}^d((x_i+\beta)\ln\frac{x_i+\beta}{y_{i}+\beta}-x_i)\\
\textrm{s.t.} \quad & \sum_{i=1}^d x_i=D\\
  &x_i\geq 0 \textrm{ for all } i=1,\ldots, d.    \\
\end{split}
\end{equation}
Using Lagrange multipliers for $x\in\Rd$, $\lambda \in \mathbb{R}$ and $\nu \in \Rd_+$
\[
\mathcal{L}(x,\lambda,\nu)=\sum_{i=1}^d((x_i+\beta)\ln\frac{x_i+\beta}{y_{i}+\beta}-x_i)-\nu^\top x-\lambda(D-\sum_{i=1}^d x_{i}).
\]
Setting $\frac{\partial \mathcal{L}}{\partial x_i}=0$, we obtain
\[
\ln\frac{x_i+\beta}{y_{i}+\beta}=\nu_i-\lambda.
\]
From the complementary slackness, we have $\nu_i=0$ for $x_i\neq 0$, which implies 
\[
x_i+\beta=\frac{1}{z}(y_{i}+\beta),
\]
where $z=\exp(\lambda)$. Let $x^*$ be the minimiser and $\mathcal{I}=\{i:x^*_i>0\}$ the support of $x^*$. Then we have 
\[
D+\abs{\mathcal{I}}\beta=\frac{1}{z}(\sum_{i\in \mathcal{I}}y_{i}+\abs{\mathcal{I}}\beta).
\]
Let $p$ be a permutation of $\{1,\ldots,d\}$ such that $y_{p(i)}\leq y_{p(i+1)}$. Define 
\[
\theta(j)=y_{p(j)}(D+(d-j+1)\beta)+\beta D-\beta\sum_{i\geq j}y_{p(i)}.
\]
It follows from
\[
\theta(j+1)-\theta(j)=(y_{p(j+1)}-y_{p(j)})(D+(d-j+1)\beta)\geq 0
\]
that $\theta(j)$ is increasing in $j$.
Let $\rho=\min \{i\lvert\theta(i) >0\}$. For all $j<\rho$, $p(j)$ is not in the support $\mathcal{I}$, since otherwise it would imply $x^*_{p(j)}\leq 0$. 
Thus the minimisation problem \eqref{min:simplex} is equivalent to
\begin{equation}
\label{min:support}
\begin{split}
\min_{x\in\Rd}\quad &\sum_{i=\rho}^d(x_{p(i)}+\beta)\ln\frac{x_{p(i)}+\beta}{y_{p(i)}+\beta}\\
\textrm{s.t.} \quad & \sum_{i=\rho}^d x_{p(i)}=D\\
  &x_{p(i)}> 0 \textrm{ for all } i=\rho,\ldots, d.    \\
\end{split}
\end{equation}
Define function $R:\mathbb{R}_{>0}\to \mathbb{R}, x\mapsto x\ln x$. It can be verified that $R$ is convex.
The objective function in \eqref{min:support} can be further rewritten into 
\[
\begin{split}
    &\sum_{i=\rho}^d(x_{p(i)}+\beta)\ln\frac{x_{p(i)}+\beta}{y_{p(i)}+\beta}\\
    =&\sum_{i=\rho}^d(y_{p(i)}+\beta)R(\frac{x_{p(i)}+\beta}{y_{p(i)}+\beta})\\
    \geq &\frac{1}{\sum_{i=\rho}^d(y_{p(i)}+\beta)}R(\frac{\sum_{i=\rho}^d(x_{p(i)}+\beta)}{\sum_{i=\rho}^d(y_{p(i)}+\beta)})\\
= &\frac{1}{\sum_{i=\rho}^d(y_{p(i)}+\beta)}R(\frac{D+(d-\rho+1)\beta}{\sum_{i=\rho}^d(y_{p(i)}+\beta)}),\\
\end{split}
\]
where the inequality follows from the Jensen's inequality. The minimum is attained if and only if 
$\frac{x_{p(i)}+\beta}{y_{p(i)}+\beta}$ are equal for all $i$. This is only possible when $\sigma(i)$ is in the support $\mathcal{I}$ for all $i\geq \rho$. Thus we can set $z=\frac{\sum_{i=\rho}^d(\abs{y_{p(i)}}+\beta)}{D+(d-\rho+1)\beta}$ and obtain $x^*_i=\max\{\frac{(\abs{y_i}+\beta)-\beta}{z},0\}\sgn(y_i)$
for $i=1\ldots d$, which is the claimed result.
\end{proof} 
\subsection{Proof of Corollary \ref{cor:acc}}
\begin{proposition}
\label{prop:decompose}
Let $\sequ{x}$ be any sequences and $\sequ{y}$ be the sequence produced by $y_{t+1}=\frac{a_t}{a_{1:t}}x_t+(1-\frac{a_t}{a_{1:t}})y_t$.
Choosing $a_t>0$ , we have, for all $x\in \mathcal{W}$ 
\[
	\begin{split}
	a_{1:T}\ex{f(y_{T+1})-f(x)}\leq&\ex{\mathcal{R}_{1:T}}-\sum_{t=1}^T(a_{1:t-1}\bd{l}{y_{t}}{y_{t+1}}),\\
	\end{split}
\]
with $\mathcal{R}_{1:T}=\sum_{t=1}^Ta_t(\inner{g_t}{x_t-x}+\comp(x_t)-\comp(x))$.\
\begin{proof}
It is interesting to see that the average scheme can be considered as an instance of the linear coupling introduced in \cite{allen2017linear}. For any sequence $\sequ{x}$, $\sequ{y}$ and $z_t=\frac{a_t}{a_{1:t}}x_t+(1-\frac{a_t}{a_{1:t}})y_t$,
we start the proof by bounding $a_t(f(y_{t+1})-f(x))$ as follows
		\begin{equation}
		\label{lemma:decompose:eq1}
		\begin{split} 
		&a_t(l(y_{t+1})-l(x))\\
		=&a_t(l(y_{t+1})-l(z_t)+l(z_t)-l(x))\\
		=& a_t(l(y_{t+1})-l(z_t)+\inner{\grad l(z_t)}{z_t-x}-\bd{l}{z_t}{x})\\
		=& a_t(l(y_{t+1})-l(z_t)+\inner{\grad l(z_t)}{z_t-x_t}+\inner{\grad l(z_t)}{x_t-x}-\bd{l}{z_t}{x})\\
		\end{split}		    
		\end{equation}
Denote by $\tau_t=\frac{a_t}{a_{1:t}}$ the weight. The first term of the the inequality above can be further bounded by
		\begin{equation}
		\label{lemma:decompose:eq2}
		\begin{split}
	    & a_t(l(y_{t+1})-l(z_t)+\inner{\grad l(z_t)}{z_t-x_t})\\
		=&a_t(l(y_{t+1})-l(z_t)+\frac{1-\tau_t}{\tau_t}\inner{\grad l(z_t)}{y_t-z_t})\\
		=& a_t(l(y_{t+1})-l(z_t)+(\frac{1}{\tau_t}-1)(l(y_t)-l(z_t))-(\frac{1}{\tau_t}-1)\bd{l}{y_t}{z_t})\\
		=&a_t(\frac{1}{\tau_t}-1)(l(y_t)-l(y_{t+1}))+\frac{a_t}{\tau_t}(l(y_{t+1})-l(z_t))-a_{1:t-1}\bd{l}{y_t}{z_t}.\\
		\end{split}
		\end{equation} Next, we have
	\begin{equation}
		\label{lemma:decompose:eq3}
	\begin{split}
	&\sum_{t=1}^Ta_t(\frac{1}{\tau_t}-1)(f(y_t)-f(y_{t+1}))\\
	=&\sum_{t=2}^{T}a_{1:t-1}(f(y_t)-f(y_{t+1}))\\
	=&\sum_{t=1}^{T-1}a_{t}f(y_{t+1})-a_{1:T-1}f(y_{T+1})\\
	=&\sum_{t=1}^{T}a_{t}f(y_{t+1})-a_{1:T}f(y_{T+1})\\
	= &\sum_{t=1}^{T}a_{t}(f(y_{t+1})-f(y_{T+1}))
	\end{split}	    
	\end{equation}
Combining \eqref{lemma:decompose:eq1}, \eqref{lemma:decompose:eq2} and \eqref{lemma:decompose:eq3}, we have
\[
	\begin{split}
	a_{1:T}(f(y_{T+1})-f(x))=&\sum_{t=1}^T\frac{a_t}{\tau_t}(l(y_{t+1})-l(z_t))\\
	&+\sum_{t=1}^T\inner{\grad l(z_t)}{x_t-x}\\
	&-\sum_{t=1}^T(a_{1:t-1}\bd{l}{y_t}{z_t}-a_t\bd{l}{z_t}{x}),
	\end{split}
\]
Simply setting $y_{t+1}\coloneqq z_t$ makes the first term above $0$ and implies $z_t=\frac{\sum_{s=1}^ta_sx_s}{a_{1:t}}$. Furthermore it
follows from the convexity of $\comp$
\[
\comp(y_{T+1})=\comp(\frac{\sum_{s=1}^Ta_sx_s}{a_{1:T}})\leq \sum_{t=1}^T\frac{a_t\comp(x_t)}{a_{1:T}}.
\] 
Combining the inequalities above and rearranging, we obtain
\[
	\begin{split}
	a_{1:T}(f(y_{T+1})-f(x))\leq&\sum_{t=1}^Ta_t(\inner{\grad l(z_t)}{x_t-x}+\comp(x_t)-\comp(x))\\
	&-\sum_{t=1}^T(a_{1:t-1}\bd{l}{z_{t-1}}{z_t}+a_t\bd{l}{z_t}{x})\\
	\leq&\sum_{t=1}^Ta_t(\inner{\grad l(z_t)}{x_t-x}+\comp(x_t)-\comp(x))\\
	&-\sum_{t=1}^T(a_{1:t-1}\bd{l}{z_{t-1}}{z_t})\\
	\end{split}
\]
Furthermore, we have
\[
\begin{split}
    &\ex{\sum_{t=1}^T\inner{a_t\grad l(z_t)}{x_t-x}}\\
    =&\ex{\sum_{t=1}^T\inner{a_tg_t}{x_t-x}}+\ex{\sum_{t=1}^T\inner{a_t(\grad l_t-g_t)}{x_t-x}}\\
    =&\ex{\sum_{t=1}^T\inner{a_tg_t}{x_t-x}}+\sum_{t=1}^T\ex{\inner{a_t(\grad l_t-g_t)}{x_t-x}}\\
    =&\ex{\sum_{t=1}^T\inner{a_tg_t}{x_t-x}}+\sum_{t=1}^T\ex{\ex{\inner{a_t(\grad l_t-g_t)}{x_t-x}\lvert z_t}}\\
    =&\ex{\sum_{t=1}^T\inner{a_tg_t}{x_t-x}}.\\
\end{split}
\]
Finally, we we obtain
\[
	\begin{split}
	a_{1:T}\ex{f(y_{T+1})-f(x)}\leq&\ex{\sum_{t=1}^Ta_t(\inner{g_t}{x_t-x}+\comp(x_t)-\comp(x))}\\
	&-\sum_{t=1}^T(a_{1:t-1}\bd{l}{y_{t}}{y_{t+1}}),\\
	\end{split}
\]
which is the claimed result.
\end{proof}
\end{proposition}

\begin{proof}[Proof of Corollary \ref{cor:acc}].
First of all, we have 
\begin{equation}
\label{cor:con:eq1}
	\begin{split}
	\ex{\mathcal{R}_{1:T}}\leq &c_1+c_2\ex{\sqrt{\sum_{t=1}^T\dualnorm{a_t(g_t-g_{t-1})}^2}}\\
	\leq &c_1+c_2\sqrt{\sum_{t=1}^T\ex{\dualnorm{a_t(g_t-g_{t-1})}^2}}\\
	\leq &c_1+c_2\sqrt{\sum_{t=1}^T\ex{\dualnorm{a_t(g_t-g_{t-1})}^2\lvert z_t}}.\\
	\end{split}
\end{equation}
For all $t$, we have
\begin{equation}
\label{cor:con:eq2}
\begin{split}
\ex{\dualnorm{a_t(g_t-g_{t-1})}^2\vert z_t}\leq&2a_t^2 (\ex{\dualnorm{g_t-\grad l(z_t)-g_{t-1}+\grad l(z_{t-1})}^2\lvert z_t})\\
&+2a_t^2 (\dualnorm{\grad l(z_t)-\grad l(z_{t-1})}^2).
\end{split}
\end{equation}
Since $z_{t-1}$ is fixed when $z_{t}$ is given, the first term above can be bounded by
\[
\begin{split}
&2a_t^2 (\ex{\dualnorm{g_t-\grad l(z_t)-g_{t-1}+\grad l(z_{t-1})}^2\lvert z_t})\\
\leq &4a_t^2 (\ex{\dualnorm{g_t-\grad l(z_t)}^2\lvert z_t}+\ex{\dualnorm{g_{t-1}-\grad l(z_{t-1})}^2\lvert z_t})\\
\leq &4a_t^2 (\ex{\dualnorm{g_t-\grad l(z_t)}^2\lvert z_t}+\ex{\dualnorm{g_{t-1}-\grad l(z_{t-1})}^2\lvert z_{t-1}})\\
\leq &4a_t^2 (\nu_t^2+\nu_{t-1}^2).\\
\end{split}
\]
Since $\cK$ is compact, there is some $L>0$ such that $\dualnorm{\grad l(z)}\leq L$ for all $z\in\mathbb{X}$. Thus the second term of \eqref{cor:con:eq2} can be bounded by
\begin{equation}
\label{cor:con:eq3}
\begin{split}
2a_t^2 \dualnorm{\grad l(z_t)-\grad l(z_{t-1})}^2\leq 8a_t^2L^2
\end{split}
\end{equation}
Combining \eqref{cor:con:eq1},  \eqref{cor:con:eq2} and  \eqref{cor:con:eq3}, we have
\[
	\ex{\mathcal{R}_{1:T}}\leq c1+c2\sqrt{8\sum_{t=1}^Ta_t^2(\nu_t^2+L^2)},
\]
and combining with Proposition \ref{prop:decompose}, we obtain
\[
	\begin{split}
	\ex{f(z_{T})-f(x)}\leq&\frac{c1+c2\sqrt{8\sum_{t=1}^Ta_t^2(\nu_t^2+L^2)}}{a_{1:T}}.\\
	\end{split}
\]
If $l$ is $M$-smooth, then for $t\geq 2$, we have
\begin{equation}
\label{cor:con:eq4}
\begin{split}
2a_t^2 \dualnorm{\grad l(z_t)-\grad l(z_{t-1})}^2\leq& \frac{4Ma_t^2}{a_{1:t-1}}a_{1:t-1}\bd{l}{z_{t-1}}{z_t}.\\
& 8Ma_{1:t-1}\bd{l}{z_{t-1}}{z_t}.\\
\end{split}
\end{equation}
Using fact $2ab-a^2\leq b^2$, we have
\begin{equation}
\label{cor:con:eq5}
\begin{split}
&2c_2\sqrt{2M\sum_{t=2}^Ta_{1:t-1}\bd{l}{z_{t-1}}{z_t}}-\sum_{t=2}^Ta_{1:t-1}\bd{l}{z_{t-1}}{z_t}\\
\leq &2Mc_2^2.
\end{split}
\end{equation}
Combining \eqref{cor:con:eq1},  \eqref{cor:con:eq2} and  \eqref{cor:con:eq5}, we have
\[
\begin{split}
	&\ex{\mathcal{R}_{1:T}}-\sum_{t=1}^Ta_{1:t-1}\bd{l}{z_{t-1}}{z_t}\\
	\leq& c_1+c_2\sqrt{\sum_{t=1}^T\ex{\dualnorm{a_t(g_t-g_{t-1})}^2\lvert z_t}}-\sum_{t=1}^Ta_{1:t-1}\bd{l}{z_{t-1}}{z_t}\\
	\leq&  c_11+c_2\sqrt{8\sum_{t=1}^Ta_t^2(\nu_t^2)}\\
	&+c_2\sqrt{\sum_{t=1}^T2a_t^2 \dualnorm{\grad l(z_t)-\grad l(z_{t-1})}^2}-\sum_{t=1}^Ta_{1:t-1}\bd{l}{z_{t-1}}{z_t}\\
	\leq&  c_11+c_2\sqrt{8\sum_{t=1}^Ta_t^2(\nu_t^2)}+c_2\sqrt{2}\dualnorm{\grad l(z_1)}\\
	&+c_2\sqrt{\sum_{t=2}^T2a_t^2 \dualnorm{\grad l(z_t)-\grad l(z_{t-1})}^2}-\sum_{t=2}^Ta_{1:t-1}\bd{l}{z_{t-1}}{z_t}\\
	\leq&  c_1+c_2\sqrt{8\sum_{t=1}^Ta_t^2(\nu_t^2)}+\sqrt{2}c_2L+2Mc_2^2,\\
\end{split}
\]
which implies
\[
	\begin{split}
	\ex{f(z_{T})-f(x)}\leq&\frac{ c_1+c_2\sqrt{8\sum_{t=1}^Ta_t^2\nu_t^2}+\sqrt{2}c_2L+2Mc_2^2}{a_{1:T}}.\\
	\end{split}
\]
\end{proof}

\section{Technical Lemmata}
\begin{lemma}
	\label{lemma:log}
	For positive values $a_1,\ldots,a_n$ the following holds:
	\begin{enumerate}
	\item \[\sum_{i=1}^{n}\frac{a_i}{\sum_{k=1}^{i}a_k+1}\leq \log (\sum_{i=1}^{n}a_i+1)\]
	\item \[\sqrt{\sum_{i=1}^{n}a_i}\leq\sum_{i=1}^{n}\frac{a_i}{\sqrt{\sum_{j=1}^ia_j^2}}\leq 2\sqrt{\sum_{i=1}^{n}a_i}.\]
	\end{enumerate}
	\begin{proof}
    The proof of (1) can be found in Lemma A.2 in \cite{levy2018online}
	For (2), we define $A_0=1$ and $A_i=\sum_{k=1}^{i}a_i+1$ for $i>0$. 
	Then we have 
		\[
		\begin{split}
		\sum_{i=1}^{n}\frac{a_i}{\sum_{k=1}^{i}a_k+1}=&\sum_{i=1}^{n}\frac{A_{i}-A_{i-1}}{A_i}\\
		=&\sum_{i=1}^{n}(1-\frac{A_{i-1}}{A_i})\\
		\leq &\sum_{i=1}^{n}\ln\frac{A_{i}}{A_{i-1}}\\
		= & \ln A_n-\ln A_0\\
		= &\ln \sum_{i=1}^{n}(a_i+1),
		\end{split}
		\]
	where the inequality follows from the concavity of $\log$.
	\end{proof}
\end{lemma}
\begin{lemma}
\label{lemma:smooth}
Let $l$ be convex and $M$-smooth over $\mathbb{X}$, i.e. 
\[
l(x)\leq l(y)+\inner{\grad l(y)}{x-y}+\frac{M}{2}\norm{x-y}^2.
\]
Then 
\[
\dualnorm{\grad l(x)-\grad l(y)}^2\leq 2M\bd{l}{x}{y}
\]
holds for all $x,y\in\mathbb{X}$.
\begin{proof}
Let $x,y\in\mathbb{X}$ be arbitrary. Define $h:\mathbb{X}\to \mathbb{R}, z\mapsto l(z)-\inner{\grad l(y)}{z}$. Clearly, $h$ is $M$-smooth and minimised at $y$. Thus we have
\[
\begin{split}
    h(y)= &\min_{z\in\mathbb{X}}h(z)\\
    \leq &\min_{z\in\mathbb{X}}h(x)+\inner{\grad h(x)}{z-x}+\frac{M}{2}\norm{z-x}^2\\
    \leq & \min_{\gamma\geq 0}h(x)-\dualnorm{\grad h(x)}\gamma+\frac{M}{2}\gamma^2\\
    =&h(x)-\frac{1}{2M}\dualnorm{\grad h(x)}^2,
\end{split}
\]
where the first inequality uses the $M$-smoothness of $h$, and the second uses $\inner{\grad h(x)}{z-x}\geq-\dualnorm{\grad h(x)}\norm{z-x}$, for which we choose $z$ such that the equality holds.
This implies
\[
\begin{split}
\frac{1}{2M}\dualnorm{\grad l(x)-\grad l(y)}^2\leq& l(x)-l(y)-\inner{\grad l(y)}{x-y}=\bd{l}{x}{y},
\end{split}
\]
and the desired result follows.
\end{proof}
\end{lemma}
\begin{lemma}
\label{lemma:bd_upper}
Define $\psi:\Rd\to \mathbb{R}, x\mapsto\sum_{i=1}^d\scf(x_i)$ for $\scf$ be as defined in \eqref{eq:entropy}. Assume $\norm{x}_1\leq D$ for all $x\in\cK\subseteq \Rd$. Setting $\beta=\frac{1}{d}$, we obtain for all $x,y\in\cK$
\[
\bd{\psi}{x}{y}\leq 4D(\ln(D+1)+\ln d).
\]
Similarly, we define $\Psi:\mathbb{R}^{m,n}\to \mathbb{R}, x\mapsto \psi\circ\sigma(x)$. Assume $\norm{x}_1\leq D$ for all $x\in \cK\subseteq \mathbb{R}^{m,n}$. Setting $\beta=\frac{1}{\min\{m,n\}}$, we obtain for all $x,y\in\cK$
\[
\bd{\Psi}{x}{y}\leq 4D(\ln(D+1)+\ln \min\{m,n\}).
\]
\begin{proof}
From the definition of the Bregman divergence it follows for all $x,y\in\cK$
\[
\begin{split}
\bd{\psi}{x}{y}=&\psi(x)-\psi(y)-\inner{\grad \psi(y)}{x-y}\\
\leq & \inner{\grad \psi(x)-\grad \psi(y)}{x-y}\\
\leq & \norm{\grad\psi(x)-\grad\psi(y)}_\infty\norm{x-y}_1\\
\leq & 2D(\norm{\grad \psi(x)}_\infty+\norm{\grad \psi(y)}_\infty).
\end{split}
\]
Using the closed form of $\norm{\grad \psi(x)}_\infty$, we have for $x\in\cK$
\[
\begin{split}
\norm{\grad \psi(x)}_\infty =&\max_i\abs{\ln(\frac{\abs{x_i}}{\beta}+1)}\\
\leq &\abs{\ln (D+\beta)}+\abs{\ln(\frac{1}{\beta})}\\
\leq & \ln (D+1)+\ln d.
\end{split}
\]
Combining the inequalities above and choosing $\beta= \frac{1}{d}$, we obtain
\[
\begin{split}
\bd{\psi}{x}{y}=&4D(\ln(D+1)+\ln d).
\end{split}
\]
Using the same argument, we have for all $x,y\in\cK\subseteq \mathbb{R}^{m,n}$
\[
\begin{split}
\bd{\Psi}{x}{y}=& 2D(\norm{\grad \Psi(x)}_\infty+\norm{\grad \Psi(y)}_\infty).
\end{split}
\]
From the characterisation of subgradient, it follows for $x\in\cK$
\[
\begin{split}
\norm{\grad \Psi(x)}_\infty = & \norm{\grad \scf(\sigma(x))}_\infty\\
\leq & \ln (D+1)+\ln \frac{1}{\beta}.
\end{split}
\]
Combine the inequalities above and choose $\beta= \frac{1}{\min\{m,n\}}$, we obtain
\[
\begin{split}
\bd{\Psi}{x}{y}\leq &4D(\ln(D+1)+\ln  \min\{m,n\}).
\end{split}
\]
\end{proof}
\end{lemma}




\end{appendices}
\bibliography{lib}

\begin{thebibliography}{}
\providecommand{\doi}[1]{\url{https://doi.org/#1}}
\bibcommenthead

\bibitem [\protect \citeauthoryear {%
Alacaoglu%
, Malitsky%
, Mertikopoulos%
\BCBL {}\ \BBA {} Cevher%
}{%
Alacaoglu%
\ \protect \BOthers {.}}{%
{\protect \APACyear {2020}}%
}]{%
alacaoglu2020new}
\APACinsertmetastar {%
alacaoglu2020new}%
\begin{APACrefauthors}%
Alacaoglu, A.%
, Malitsky, Y.%
, Mertikopoulos, P.%
\BCBL {} Cevher, V.%
\end{APACrefauthors}%
\unskip\
\newblock
\APACrefYearMonthDay{2020}{}{}.
\newblock
{\BBOQ}\APACrefatitle {A new regret analysis for Adam-type algorithms} {A new
  regret analysis for adam-type algorithms}.{\BBCQ}
\newblock
 \APACrefbtitle {International Conference on Machine Learning} {International
  conference on machine learning}\ (\BPGS\ 202--210).
\PrintBackRefs{\CurrentBib}

\bibitem [\protect \citeauthoryear {%
Allen-Zhu%
\ \BBA {} Orecchia%
}{%
Allen-Zhu%
\ \BBA {} Orecchia%
}{%
{\protect \APACyear {2017}}%
}]{%
allen2017linear}
\APACinsertmetastar {%
allen2017linear}%
\begin{APACrefauthors}%
Allen-Zhu, Z.%
\BCBT {}\ \BBA {} Orecchia, L.%
\end{APACrefauthors}%
\unskip\
\newblock
\APACrefYearMonthDay{2017}{}{}.
\newblock
{\BBOQ}\APACrefatitle {Linear Coupling: An Ultimate Unification of Gradient and
  Mirror Descent} {Linear coupling: An ultimate unification of gradient and
  mirror descent}.{\BBCQ}
\newblock
 \APACrefbtitle {8th Innovations in Theoretical Computer Science Conference
  (ITCS 2017).} {8th innovations in theoretical computer science conference
  (itcs 2017).}
\PrintBackRefs{\CurrentBib}

\bibitem [\protect \citeauthoryear {%
Anava%
, Hazan%
, Mannor%
\BCBL {}\ \BBA {} Shamir%
}{%
Anava%
\ \protect \BOthers {.}}{%
{\protect \APACyear {2013}}%
}]{%
anava2013online}
\APACinsertmetastar {%
anava2013online}%
\begin{APACrefauthors}%
Anava, O.%
, Hazan, E.%
, Mannor, S.%
\BCBL {} Shamir, O.%
\end{APACrefauthors}%
\unskip\
\newblock
\APACrefYearMonthDay{2013}{}{}.
\newblock
{\BBOQ}\APACrefatitle {Online learning for time series prediction} {Online
  learning for time series prediction}.{\BBCQ}
\newblock
 \APACrefbtitle {Conference on learning theory} {Conference on learning
  theory}\ (\BPGS\ 172--184).
\PrintBackRefs{\CurrentBib}

\bibitem [\protect \citeauthoryear {%
Arora%
, Hazan%
\BCBL {}\ \BBA {} Kale%
}{%
Arora%
\ \protect \BOthers {.}}{%
{\protect \APACyear {2012}}%
}]{%
arora2012multiplicative}
\APACinsertmetastar {%
arora2012multiplicative}%
\begin{APACrefauthors}%
Arora, S.%
, Hazan, E.%
\BCBL {} Kale, S.%
\end{APACrefauthors}%
\unskip\
\newblock
\APACrefYearMonthDay{2012}{}{}.
\newblock
{\BBOQ}\APACrefatitle {The multiplicative weights update method: a
  meta-algorithm and applications} {The multiplicative weights update method: a
  meta-algorithm and applications}.{\BBCQ}
\newblock
\APACjournalVolNumPages{Theory of Computing}{8}{1}{121--164}.
\newblock

\newblock

\PrintBackRefs{\CurrentBib}

\bibitem [\protect \citeauthoryear {%
Barbu%
\ \BBA {} Precupanu%
}{%
Barbu%
\ \BBA {} Precupanu%
}{%
{\protect \APACyear {2012}}%
}]{%
barbu2012convexity}
\APACinsertmetastar {%
barbu2012convexity}%
\begin{APACrefauthors}%
Barbu, V.%
\BCBT {}\ \BBA {} Precupanu, T.%
\end{APACrefauthors}%
\unskip\
\newblock
\APACrefYear{2012}.
\newblock
\APACrefbtitle {Convexity and optimization in Banach spaces} {Convexity and
  optimization in banach spaces}.
\newblock
\APACaddressPublisher{}{Springer Science \& Business Media}.
\PrintBackRefs{\CurrentBib}

\bibitem [\protect \citeauthoryear {%
Beck%
\ \BBA {} Teboulle%
}{%
Beck%
\ \BBA {} Teboulle%
}{%
{\protect \APACyear {2009}}%
}]{%
beck2009fast}
\APACinsertmetastar {%
beck2009fast}%
\begin{APACrefauthors}%
Beck, A.%
\BCBT {}\ \BBA {} Teboulle, M.%
\end{APACrefauthors}%
\unskip\
\newblock
\APACrefYearMonthDay{2009}{}{}.
\newblock
{\BBOQ}\APACrefatitle {A fast iterative shrinkage-thresholding algorithm for
  linear inverse problems} {A fast iterative shrinkage-thresholding algorithm
  for linear inverse problems}.{\BBCQ}
\newblock
\APACjournalVolNumPages{SIAM journal on imaging sciences}{2}{1}{183--202}.
\newblock

\newblock

\PrintBackRefs{\CurrentBib}

\bibitem [\protect \citeauthoryear {%
Bhatia%
}{%
Bhatia%
}{%
{\protect \APACyear {2013}}%
}]{%
bhatia2013matrix}
\APACinsertmetastar {%
bhatia2013matrix}%
\begin{APACrefauthors}%
Bhatia, R.%
\end{APACrefauthors}%
\unskip\
\newblock
\APACrefYear{2013}.
\newblock
\APACrefbtitle {Matrix analysis} {Matrix analysis}\ (\BVOL~169).
\newblock
\APACaddressPublisher{}{Springer Science \& Business Media}.
\PrintBackRefs{\CurrentBib}

\bibitem [\protect \citeauthoryear {%
Cancela%
, Bolón-Canedo%
\BCBL {}\ \BBA {} Alonso-Betanzos%
}{%
Cancela%
\ \protect \BOthers {.}}{%
{\protect \APACyear {2021}}%
}]{%
9413170}
\APACinsertmetastar {%
9413170}%
\begin{APACrefauthors}%
Cancela, B.%
, Bolón-Canedo, V.%
\BCBL {} Alonso-Betanzos, A.%
\end{APACrefauthors}%
\unskip\
\newblock
\APACrefYearMonthDay{2021}{}{}.
\newblock
{\BBOQ}\APACrefatitle {A delayed Elastic-Net approach for performing
  adversarial attacks} {A delayed elastic-net approach for performing
  adversarial attacks}.{\BBCQ}
\newblock
 \APACrefbtitle {2020 25th International Conference on Pattern Recognition
  (ICPR)} {2020 25th international conference on pattern recognition (icpr)}\
  (\BPG~378-384).
\newblock
\begin{APACrefDOI} \doi{10.1109/ICPR48806.2021.9413170} \end{APACrefDOI}
\PrintBackRefs{\CurrentBib}

\bibitem [\protect \citeauthoryear {%
Carlini%
\ \BBA {} Wagner%
}{%
Carlini%
\ \BBA {} Wagner%
}{%
{\protect \APACyear {2017}}%
}]{%
carlini2017towards}
\APACinsertmetastar {%
carlini2017towards}%
\begin{APACrefauthors}%
Carlini, N.%
\BCBT {}\ \BBA {} Wagner, D.%
\end{APACrefauthors}%
\unskip\
\newblock
\APACrefYearMonthDay{2017}{}{}.
\newblock
{\BBOQ}\APACrefatitle {Towards evaluating the robustness of neural networks}
  {Towards evaluating the robustness of neural networks}.{\BBCQ}
\newblock
 \APACrefbtitle {2017 ieee symposium on security and privacy (sp)} {2017 ieee
  symposium on security and privacy (sp)}\ (\BPGS\ 39--57).
\PrintBackRefs{\CurrentBib}

\bibitem [\protect \citeauthoryear {%
Cesa-Bianchi%
, Conconi%
\BCBL {}\ \BBA {} Gentile%
}{%
Cesa-Bianchi%
\ \protect \BOthers {.}}{%
{\protect \APACyear {2004}}%
}]{%
cesa2004generalization}
\APACinsertmetastar {%
cesa2004generalization}%
\begin{APACrefauthors}%
Cesa-Bianchi, N.%
, Conconi, A.%
\BCBL {} Gentile, C.%
\end{APACrefauthors}%
\unskip\
\newblock
\APACrefYearMonthDay{2004}{}{}.
\newblock
{\BBOQ}\APACrefatitle {On the generalization ability of on-line learning
  algorithms} {On the generalization ability of on-line learning
  algorithms}.{\BBCQ}
\newblock
\APACjournalVolNumPages{IEEE Transactions on Information
  Theory}{50}{9}{2050--2057}.
\newblock

\newblock

\PrintBackRefs{\CurrentBib}

\bibitem [\protect \citeauthoryear {%
Cesa-Bianchi%
\ \BBA {} Gentile%
}{%
Cesa-Bianchi%
\ \BBA {} Gentile%
}{%
{\protect \APACyear {2008}}%
}]{%
cesa2008improved}
\APACinsertmetastar {%
cesa2008improved}%
\begin{APACrefauthors}%
Cesa-Bianchi, N.%
\BCBT {}\ \BBA {} Gentile, C.%
\end{APACrefauthors}%
\unskip\
\newblock
\APACrefYearMonthDay{2008}{}{}.
\newblock
{\BBOQ}\APACrefatitle {Improved risk tail bounds for on-line algorithms}
  {Improved risk tail bounds for on-line algorithms}.{\BBCQ}
\newblock
\APACjournalVolNumPages{IEEE Transactions on Information
  Theory}{54}{1}{386--390}.
\newblock

\newblock

\PrintBackRefs{\CurrentBib}

\bibitem [\protect \citeauthoryear {%
P\BHBI Y.~Chen%
, Sharma%
, Zhang%
, Yi%
\BCBL {}\ \BBA {} Hsieh%
}{%
P\BHBI Y.~Chen%
\ \protect \BOthers {.}}{%
{\protect \APACyear {2018}}%
}]{%
chen2018ead}
\APACinsertmetastar {%
chen2018ead}%
\begin{APACrefauthors}%
Chen, P\BHBI Y.%
, Sharma, Y.%
, Zhang, H.%
, Yi, J.%
\BCBL {} Hsieh, C\BHBI J.%
\end{APACrefauthors}%
\unskip\
\newblock
\APACrefYearMonthDay{2018}{}{}.
\newblock
{\BBOQ}\APACrefatitle {Ead: elastic-net attacks to deep neural networks via
  adversarial examples} {Ead: elastic-net attacks to deep neural networks via
  adversarial examples}.{\BBCQ}
\newblock
 \APACrefbtitle {Thirty-second AAAI conference on artificial intelligence.}
  {Thirty-second aaai conference on artificial intelligence.}
\PrintBackRefs{\CurrentBib}

\bibitem [\protect \citeauthoryear {%
X.~Chen%
\ \protect \BOthers {.}}{%
X.~Chen%
\ \protect \BOthers {.}}{%
{\protect \APACyear {2019}}%
}]{%
chen2019zo}
\APACinsertmetastar {%
chen2019zo}%
\begin{APACrefauthors}%
Chen, X.%
, Liu, S.%
, Xu, K.%
, Li, X.%
, Lin, X.%
, Hong, M.%
\BCBL {} Cox, D.%
\end{APACrefauthors}%
\unskip\
\newblock
\APACrefYearMonthDay{2019}{}{}.
\newblock
{\BBOQ}\APACrefatitle {Zo-adamm: Zeroth-order adaptive momentum method for
  black-box optimization} {Zo-adamm: Zeroth-order adaptive momentum method for
  black-box optimization}.{\BBCQ}
\newblock
\APACjournalVolNumPages{Advances in Neural Information Processing
  Systems}{32}{}{}.
\newblock

\newblock

\PrintBackRefs{\CurrentBib}

\bibitem [\protect \citeauthoryear {%
Cutkosky%
}{%
Cutkosky%
}{%
{\protect \APACyear {2019}}%
}]{%
cutkosky2019anytime}
\APACinsertmetastar {%
cutkosky2019anytime}%
\begin{APACrefauthors}%
Cutkosky, A.%
\end{APACrefauthors}%
\unskip\
\newblock
\APACrefYearMonthDay{2019}{}{}.
\newblock
{\BBOQ}\APACrefatitle {Anytime online-to-batch, optimism and acceleration}
  {Anytime online-to-batch, optimism and acceleration}.{\BBCQ}
\newblock
 \APACrefbtitle {International Conference on Machine Learning} {International
  conference on machine learning}\ (\BPGS\ 1446--1454).
\PrintBackRefs{\CurrentBib}

\bibitem [\protect \citeauthoryear {%
Cutkosky%
\ \BBA {} Boahen%
}{%
Cutkosky%
\ \BBA {} Boahen%
}{%
{\protect \APACyear {2017}}%
{\protect \APACexlab {{\protect \BCnt {1}}}}}]{%
cutkosky2017online}
\APACinsertmetastar {%
cutkosky2017online}%
\begin{APACrefauthors}%
Cutkosky, A.%
\BCBT {}\ \BBA {} Boahen, K.%
\end{APACrefauthors}%
\unskip\
\newblock
\APACrefYearMonthDay{2017{\protect \BCnt {1}}}{}{}.
\newblock
{\BBOQ}\APACrefatitle {Online learning without prior information} {Online
  learning without prior information}.{\BBCQ}
\newblock
 \APACrefbtitle {Conference on Learning Theory} {Conference on learning
  theory}\ (\BPGS\ 643--677).
\PrintBackRefs{\CurrentBib}

\bibitem [\protect \citeauthoryear {%
Cutkosky%
\ \BBA {} Boahen%
}{%
Cutkosky%
\ \BBA {} Boahen%
}{%
{\protect \APACyear {2016}}%
}]{%
cutkosky2016online}
\APACinsertmetastar {%
cutkosky2016online}%
\begin{APACrefauthors}%
Cutkosky, A.%
\BCBT {}\ \BBA {} Boahen, K.A.%
\end{APACrefauthors}%
\unskip\
\newblock
\APACrefYearMonthDay{2016}{}{}.
\newblock
{\BBOQ}\APACrefatitle {Online convex optimization with unconstrained domains
  and losses} {Online convex optimization with unconstrained domains and
  losses}.{\BBCQ}
\newblock
\APACjournalVolNumPages{Advances in Neural Information Processing
  Systems}{29}{}{}.
\newblock

\newblock

\PrintBackRefs{\CurrentBib}

\bibitem [\protect \citeauthoryear {%
Cutkosky%
\ \BBA {} Boahen%
}{%
Cutkosky%
\ \BBA {} Boahen%
}{%
{\protect \APACyear {2017}}%
{\protect \APACexlab {{\protect \BCnt {2}}}}}]{%
cutkosky2017stochastic}
\APACinsertmetastar {%
cutkosky2017stochastic}%
\begin{APACrefauthors}%
Cutkosky, A.%
\BCBT {}\ \BBA {} Boahen, K.A.%
\end{APACrefauthors}%
\unskip\
\newblock
\APACrefYearMonthDay{2017{\protect \BCnt {2}}}{}{}.
\newblock
{\BBOQ}\APACrefatitle {Stochastic and adversarial online learning without
  hyperparameters} {Stochastic and adversarial online learning without
  hyperparameters}.{\BBCQ}
\newblock
\APACjournalVolNumPages{Advances in Neural Information Processing
  Systems}{30}{}{}.
\newblock

\newblock

\PrintBackRefs{\CurrentBib}

\bibitem [\protect \citeauthoryear {%
Dhurandhar%
\ \protect \BOthers {.}}{%
Dhurandhar%
\ \protect \BOthers {.}}{%
{\protect \APACyear {2018}}%
}]{%
NEURIPS2018_c5ff2543}
\APACinsertmetastar {%
NEURIPS2018_c5ff2543}%
\begin{APACrefauthors}%
Dhurandhar, A.%
, Chen, P\BHBI Y.%
, Luss, R.%
, Tu, C\BHBI C.%
, Ting, P.%
, Shanmugam, K.%
\BCBL {} Das, P.%
\end{APACrefauthors}%
\unskip\
\newblock
\APACrefYearMonthDay{2018}{}{}.
\newblock
{\BBOQ}\APACrefatitle {Explanations based on the Missing: Towards Contrastive
  Explanations with Pertinent Negatives} {Explanations based on the missing:
  Towards contrastive explanations with pertinent negatives}.{\BBCQ}
\newblock
 S.~Bengio, H.~Wallach, H.~Larochelle, K.~Grauman, N.~Cesa-Bianchi\BCBL {}\
  \BBA {} R.~Garnett\ (\BEDS), \APACrefbtitle {Advances in Neural Information
  Processing Systems} {Advances in neural information processing systems}\
  (\BVOL~31).
\newblock
\APACaddressPublisher{}{Curran Associates, Inc.}
\PrintBackRefs{\CurrentBib}

\bibitem [\protect \citeauthoryear {%
J.~Duchi%
, Hazan%
\BCBL {}\ \BBA {} Singer%
}{%
J.~Duchi%
\ \protect \BOthers {.}}{%
{\protect \APACyear {2011}}%
}]{%
duchi2011adaptive}
\APACinsertmetastar {%
duchi2011adaptive}%
\begin{APACrefauthors}%
Duchi, J.%
, Hazan, E.%
\BCBL {} Singer, Y.%
\end{APACrefauthors}%
\unskip\
\newblock
\APACrefYearMonthDay{2011}{}{}.
\newblock
{\BBOQ}\APACrefatitle {Adaptive subgradient methods for online learning and
  stochastic optimization} {Adaptive subgradient methods for online learning
  and stochastic optimization}.{\BBCQ}
\newblock
\APACjournalVolNumPages{Journal of Machine Learning
  Research}{12}{Jul}{2121--2159}.
\newblock

\newblock

\PrintBackRefs{\CurrentBib}

\bibitem [\protect \citeauthoryear {%
J.C.~Duchi%
, Jordan%
, Wainwright%
\BCBL {}\ \BBA {} Wibisono%
}{%
J.C.~Duchi%
\ \protect \BOthers {.}}{%
{\protect \APACyear {2015}}%
}]{%
duchi2015optimal}
\APACinsertmetastar {%
duchi2015optimal}%
\begin{APACrefauthors}%
Duchi, J.C.%
, Jordan, M.I.%
, Wainwright, M.J.%
\BCBL {} Wibisono, A.%
\end{APACrefauthors}%
\unskip\
\newblock
\APACrefYearMonthDay{2015}{}{}.
\newblock
{\BBOQ}\APACrefatitle {Optimal rates for zero-order convex optimization: The
  power of two function evaluations} {Optimal rates for zero-order convex
  optimization: The power of two function evaluations}.{\BBCQ}
\newblock
\APACjournalVolNumPages{IEEE Transactions on Information
  Theory}{61}{5}{2788--2806}.
\newblock

\newblock

\PrintBackRefs{\CurrentBib}

\bibitem [\protect \citeauthoryear {%
J.C.~Duchi%
, Shalev{-}Shwartz%
, Singer%
\BCBL {}\ \BBA {} Tewari%
}{%
J.C.~Duchi%
\ \protect \BOthers {.}}{%
{\protect \APACyear {2010}}%
}]{%
duchi2010composite}
\APACinsertmetastar {%
duchi2010composite}%
\begin{APACrefauthors}%
Duchi, J.C.%
, Shalev{-}Shwartz, S.%
, Singer, Y.%
\BCBL {} Tewari, A.%
\end{APACrefauthors}%
\unskip\
\newblock
\APACrefYearMonthDay{2010}{}{}.
\newblock
{\BBOQ}\APACrefatitle {Composite Objective Mirror Descent} {Composite objective
  mirror descent}.{\BBCQ}
\newblock
 A.T.~Kalai\ \BBA {} M.~Mohri\ (\BEDS), \APACrefbtitle {{COLT} 2010 - The 23rd
  Conference on Learning Theory, Haifa, Israel, June 27-29, 2010} {{COLT} 2010
  - the 23rd conference on learning theory, haifa, israel, june 27-29, 2010}\
  (\BPGS\ 14--26).
\newblock
\APACaddressPublisher{}{Omnipress}.
\PrintBackRefs{\CurrentBib}

\bibitem [\protect \citeauthoryear {%
Gentile%
}{%
Gentile%
}{%
{\protect \APACyear {2003}}%
}]{%
gentile2003robustness}
\APACinsertmetastar {%
gentile2003robustness}%
\begin{APACrefauthors}%
Gentile, C.%
\end{APACrefauthors}%
\unskip\
\newblock
\APACrefYearMonthDay{2003}{}{}.
\newblock
{\BBOQ}\APACrefatitle {The robustness of the p-norm algorithms} {The robustness
  of the p-norm algorithms}.{\BBCQ}
\newblock
\APACjournalVolNumPages{Machine Learning}{53}{3}{265--299}.
\newblock

\newblock

\PrintBackRefs{\CurrentBib}

\bibitem [\protect \citeauthoryear {%
Ghai%
, Hazan%
\BCBL {}\ \BBA {} Singer%
}{%
Ghai%
\ \protect \BOthers {.}}{%
{\protect \APACyear {2020}}%
}]{%
ghai2020exponentiated}
\APACinsertmetastar {%
ghai2020exponentiated}%
\begin{APACrefauthors}%
Ghai, U.%
, Hazan, E.%
\BCBL {} Singer, Y.%
\end{APACrefauthors}%
\unskip\
\newblock
\APACrefYearMonthDay{2020}{}{}.
\newblock
{\BBOQ}\APACrefatitle {Exponentiated gradient meets gradient descent}
  {Exponentiated gradient meets gradient descent}.{\BBCQ}
\newblock
 \APACrefbtitle {Algorithmic Learning Theory} {Algorithmic learning theory}\
  (\BPGS\ 386--407).
\PrintBackRefs{\CurrentBib}

\bibitem [\protect \citeauthoryear {%
He%
, Zhang%
, Ren%
\BCBL {}\ \BBA {} Sun%
}{%
He%
\ \protect \BOthers {.}}{%
{\protect \APACyear {2016}}%
}]{%
7780459}
\APACinsertmetastar {%
7780459}%
\begin{APACrefauthors}%
He, K.%
, Zhang, X.%
, Ren, S.%
\BCBL {} Sun, J.%
\end{APACrefauthors}%
\unskip\
\newblock
\APACrefYearMonthDay{2016}{}{}.
\newblock
{\BBOQ}\APACrefatitle {Deep Residual Learning for Image Recognition} {Deep
  residual learning for image recognition}.{\BBCQ}
\newblock
 \APACrefbtitle {2016 IEEE Conference on Computer Vision and Pattern
  Recognition (CVPR)} {2016 ieee conference on computer vision and pattern
  recognition (cvpr)}\ (\BPG~770-778).
\newblock
\begin{APACrefDOI} \doi{10.1109/CVPR.2016.90} \end{APACrefDOI}
\PrintBackRefs{\CurrentBib}

\bibitem [\protect \citeauthoryear {%
Joulani%
, Gy{\"o}rgy%
\BCBL {}\ \BBA {} Szepesv{\'a}ri%
}{%
Joulani%
\ \protect \BOthers {.}}{%
{\protect \APACyear {2017}}%
}]{%
joulani2017modular}
\APACinsertmetastar {%
joulani2017modular}%
\begin{APACrefauthors}%
Joulani, P.%
, Gy{\"o}rgy, A.%
\BCBL {} Szepesv{\'a}ri, C.%
\end{APACrefauthors}%
\unskip\
\newblock
\APACrefYearMonthDay{2017}{}{}.
\newblock
{\BBOQ}\APACrefatitle {A Modular Analysis of Adaptive (Non-) Convex
  Optimization: Optimism, Composite Objectives, and Variational Bounds} {A
  modular analysis of adaptive (non-) convex optimization: Optimism, composite
  objectives, and variational bounds}.{\BBCQ}
\newblock
\APACjournalVolNumPages{Journal of Machine Learning Research}{1}{}{40}.
\newblock

\newblock

\PrintBackRefs{\CurrentBib}

\bibitem [\protect \citeauthoryear {%
Joulani%
, Raj%
, Gyorgy%
\BCBL {}\ \BBA {} Szepesv{\'a}ri%
}{%
Joulani%
\ \protect \BOthers {.}}{%
{\protect \APACyear {2020}}%
}]{%
joulani2020simpler}
\APACinsertmetastar {%
joulani2020simpler}%
\begin{APACrefauthors}%
Joulani, P.%
, Raj, A.%
, Gyorgy, A.%
\BCBL {} Szepesv{\'a}ri, C.%
\end{APACrefauthors}%
\unskip\
\newblock
\APACrefYearMonthDay{2020}{}{}.
\newblock
{\BBOQ}\APACrefatitle {A simpler approach to accelerated optimization:
  iterative averaging meets optimism} {A simpler approach to accelerated
  optimization: iterative averaging meets optimism}.{\BBCQ}
\newblock
 \APACrefbtitle {International Conference on Machine Learning} {International
  conference on machine learning}\ (\BPGS\ 4984--4993).
\PrintBackRefs{\CurrentBib}

\bibitem [\protect \citeauthoryear {%
Kakade%
, Shalev-Shwartz%
\BCBL {}\ \BBA {} Tewari%
}{%
Kakade%
\ \protect \BOthers {.}}{%
{\protect \APACyear {2012}}%
}]{%
kakade2012regularization}
\APACinsertmetastar {%
kakade2012regularization}%
\begin{APACrefauthors}%
Kakade, S.M.%
, Shalev-Shwartz, S.%
\BCBL {} Tewari, A.%
\end{APACrefauthors}%
\unskip\
\newblock
\APACrefYearMonthDay{2012}{}{}.
\newblock
{\BBOQ}\APACrefatitle {Regularization techniques for learning with matrices}
  {Regularization techniques for learning with matrices}.{\BBCQ}
\newblock
\APACjournalVolNumPages{The Journal of Machine Learning
  Research}{13}{1}{1865--1890}.
\newblock

\newblock

\PrintBackRefs{\CurrentBib}

\bibitem [\protect \citeauthoryear {%
Kavis%
, Levy%
, Bach%
\BCBL {}\ \BBA {} Cevher%
}{%
Kavis%
\ \protect \BOthers {.}}{%
{\protect \APACyear {2019}}%
}]{%
kavis2019unixgrad}
\APACinsertmetastar {%
kavis2019unixgrad}%
\begin{APACrefauthors}%
Kavis, A.%
, Levy, K.Y.%
, Bach, F.%
\BCBL {} Cevher, V.%
\end{APACrefauthors}%
\unskip\
\newblock
\APACrefYearMonthDay{2019}{}{}.
\newblock
{\BBOQ}\APACrefatitle {Unixgrad: A universal, adaptive algorithm with optimal
  guarantees for constrained optimization} {Unixgrad: A universal, adaptive
  algorithm with optimal guarantees for constrained optimization}.{\BBCQ}
\newblock
 \APACrefbtitle {Advances in Neural Information Processing Systems} {Advances
  in neural information processing systems}\ (\BPGS\ 6260--6269).
\PrintBackRefs{\CurrentBib}

\bibitem [\protect \citeauthoryear {%
Kempka%
, Kotlowski%
\BCBL {}\ \BBA {} Warmuth%
}{%
Kempka%
\ \protect \BOthers {.}}{%
{\protect \APACyear {2019}}%
}]{%
kempka2019adaptive}
\APACinsertmetastar {%
kempka2019adaptive}%
\begin{APACrefauthors}%
Kempka, M.%
, Kotlowski, W.%
\BCBL {} Warmuth, M.K.%
\end{APACrefauthors}%
\unskip\
\newblock
\APACrefYearMonthDay{2019}{}{}.
\newblock
{\BBOQ}\APACrefatitle {Adaptive scale-invariant online algorithms for learning
  linear models} {Adaptive scale-invariant online algorithms for learning
  linear models}.{\BBCQ}
\newblock
 \APACrefbtitle {International Conference on Machine Learning} {International
  conference on machine learning}\ (\BPGS\ 3321--3330).
\PrintBackRefs{\CurrentBib}

\bibitem [\protect \citeauthoryear {%
Kivinen%
\ \BBA {} Warmuth%
}{%
Kivinen%
\ \BBA {} Warmuth%
}{%
{\protect \APACyear {1997}}%
}]{%
kivinen1997exponentiated}
\APACinsertmetastar {%
kivinen1997exponentiated}%
\begin{APACrefauthors}%
Kivinen, J.%
\BCBT {}\ \BBA {} Warmuth, M.K.%
\end{APACrefauthors}%
\unskip\
\newblock
\APACrefYearMonthDay{1997}{}{}.
\newblock
{\BBOQ}\APACrefatitle {Exponentiated gradient versus gradient descent for
  linear predictors} {Exponentiated gradient versus gradient descent for linear
  predictors}.{\BBCQ}
\newblock
\APACjournalVolNumPages{information and computation}{132}{1}{1--63}.
\newblock

\newblock

\PrintBackRefs{\CurrentBib}

\bibitem [\protect \citeauthoryear {%
Krizhevsky%
}{%
Krizhevsky%
}{%
{\protect \APACyear {2009}}%
}]{%
krizhevsky2009learning}
\APACinsertmetastar {%
krizhevsky2009learning}%
\begin{APACrefauthors}%
Krizhevsky, A.%
\end{APACrefauthors}%
\unskip\
\newblock
\APACrefYearMonthDay{2009}{}{}.
\newblock
{\BBOQ}\APACrefatitle {Learning Multiple Layers of Features from Tiny Images}
  {Learning multiple layers of features from tiny images}.{\BBCQ}
\newblock
\APACjournalVolNumPages{Master's thesis, University of Tront}{}{}{}.
\newblock

\newblock

\PrintBackRefs{\CurrentBib}

\bibitem [\protect \citeauthoryear {%
Lan%
}{%
Lan%
}{%
{\protect \APACyear {2020}}%
}]{%
lan2020first}
\APACinsertmetastar {%
lan2020first}%
\begin{APACrefauthors}%
Lan, G.%
\end{APACrefauthors}%
\unskip\
\newblock
\APACrefYear{2020}.
\newblock
\APACrefbtitle {First-order and stochastic optimization methods for machine
  learning} {First-order and stochastic optimization methods for machine
  learning}.
\newblock
\APACaddressPublisher{}{Springer}.
\PrintBackRefs{\CurrentBib}

\bibitem [\protect \citeauthoryear {%
Levy%
, Yurtsever%
\BCBL {}\ \BBA {} Cevher%
}{%
Levy%
\ \protect \BOthers {.}}{%
{\protect \APACyear {2018}}%
}]{%
levy2018online}
\APACinsertmetastar {%
levy2018online}%
\begin{APACrefauthors}%
Levy, Y.K.%
, Yurtsever, A.%
\BCBL {} Cevher, V.%
\end{APACrefauthors}%
\unskip\
\newblock
\APACrefYearMonthDay{2018}{}{}.
\newblock
{\BBOQ}\APACrefatitle {Online adaptive methods, universality and acceleration}
  {Online adaptive methods, universality and acceleration}.{\BBCQ}
\newblock
 \APACrefbtitle {Advances in Neural Information Processing Systems} {Advances
  in neural information processing systems}\ (\BPGS\ 6500--6509).
\PrintBackRefs{\CurrentBib}

\bibitem [\protect \citeauthoryear {%
Lewis%
}{%
Lewis%
}{%
{\protect \APACyear {1995}}%
}]{%
lewis1995convex}
\APACinsertmetastar {%
lewis1995convex}%
\begin{APACrefauthors}%
Lewis, A.S.%
\end{APACrefauthors}%
\unskip\
\newblock
\APACrefYearMonthDay{1995}{}{}.
\newblock
{\BBOQ}\APACrefatitle {The convex analysis of unitarily invariant matrix
  functions} {The convex analysis of unitarily invariant matrix
  functions}.{\BBCQ}
\newblock
\APACjournalVolNumPages{Journal of Convex Analysis}{2}{1}{173--183}.
\newblock

\newblock

\PrintBackRefs{\CurrentBib}

\bibitem [\protect \citeauthoryear {%
Li%
\ \BBA {} Orabona%
}{%
Li%
\ \BBA {} Orabona%
}{%
{\protect \APACyear {2019}}%
}]{%
li2019convergence}
\APACinsertmetastar {%
li2019convergence}%
\begin{APACrefauthors}%
Li, X.%
\BCBT {}\ \BBA {} Orabona, F.%
\end{APACrefauthors}%
\unskip\
\newblock
\APACrefYearMonthDay{2019}{}{}.
\newblock
{\BBOQ}\APACrefatitle {On the convergence of stochastic gradient descent with
  adaptive stepsizes} {On the convergence of stochastic gradient descent with
  adaptive stepsizes}.{\BBCQ}
\newblock
 \APACrefbtitle {The 22nd International Conference on Artificial Intelligence
  and Statistics} {The 22nd international conference on artificial intelligence
  and statistics}\ (\BPGS\ 983--992).
\PrintBackRefs{\CurrentBib}

\bibitem [\protect \citeauthoryear {%
Lu%
, Lin%
\BCBL {}\ \BBA {} Yan%
}{%
Lu%
\ \protect \BOthers {.}}{%
{\protect \APACyear {2014}}%
}]{%
lu2014smoothed}
\APACinsertmetastar {%
lu2014smoothed}%
\begin{APACrefauthors}%
Lu, C.%
, Lin, Z.%
\BCBL {} Yan, S.%
\end{APACrefauthors}%
\unskip\
\newblock
\APACrefYearMonthDay{2014}{}{}.
\newblock
{\BBOQ}\APACrefatitle {Smoothed low rank and sparse matrix recovery by
  iteratively reweighted least squares minimization} {Smoothed low rank and
  sparse matrix recovery by iteratively reweighted least squares
  minimization}.{\BBCQ}
\newblock
\APACjournalVolNumPages{IEEE Transactions on Image
  Processing}{24}{2}{646--654}.
\newblock

\newblock

\PrintBackRefs{\CurrentBib}

\bibitem [\protect \citeauthoryear {%
McMahan%
\ \BBA {} Streeter%
}{%
McMahan%
\ \BBA {} Streeter%
}{%
{\protect \APACyear {2010}}%
}]{%
mcmahanadaptive}
\APACinsertmetastar {%
mcmahanadaptive}%
\begin{APACrefauthors}%
McMahan, H.B.%
\BCBT {}\ \BBA {} Streeter, M.J.%
\end{APACrefauthors}%
\unskip\
\newblock
\APACrefYearMonthDay{2010}{}{}.
\newblock
{\BBOQ}\APACrefatitle {Adaptive Bound Optimization for Online Convex
  Optimization} {Adaptive bound optimization for online convex
  optimization}.{\BBCQ}
\newblock
 A.T.~Kalai\ \BBA {} M.~Mohri\ (\BEDS), \APACrefbtitle {{COLT} 2010 - The 23rd
  Conference on Learning Theory, Haifa, Israel, June 27-29, 2010} {{COLT} 2010
  - the 23rd conference on learning theory, haifa, israel, june 27-29, 2010}\
  (\BPGS\ 244--256).
\newblock
\APACaddressPublisher{}{Omnipress}.
\PrintBackRefs{\CurrentBib}

\bibitem [\protect \citeauthoryear {%
Nesterov%
}{%
Nesterov%
}{%
{\protect \APACyear {2003}}%
}]{%
nesterov2003introductory}
\APACinsertmetastar {%
nesterov2003introductory}%
\begin{APACrefauthors}%
Nesterov, Y.%
\end{APACrefauthors}%
\unskip\
\newblock
\APACrefYear{2003}.
\newblock
\APACrefbtitle {Introductory lectures on convex optimization: A basic course}
  {Introductory lectures on convex optimization: A basic course}\ (\BVOL~87).
\newblock
\APACaddressPublisher{}{Springer Science \& Business Media}.
\PrintBackRefs{\CurrentBib}

\bibitem [\protect \citeauthoryear {%
Orabona%
}{%
Orabona%
}{%
{\protect \APACyear {2013}}%
}]{%
orabona2013dimension}
\APACinsertmetastar {%
orabona2013dimension}%
\begin{APACrefauthors}%
Orabona, F.%
\end{APACrefauthors}%
\unskip\
\newblock
\APACrefYearMonthDay{2013}{}{}.
\newblock
{\BBOQ}\APACrefatitle {Dimension-Free Exponentiated Gradient.} {Dimension-free
  exponentiated gradient.}{\BBCQ}
\newblock
 \APACrefbtitle {NIPS} {Nips}\ (\BPGS\ 1806--1814).
\PrintBackRefs{\CurrentBib}

\bibitem [\protect \citeauthoryear {%
Orabona%
, Crammer%
\BCBL {}\ \BBA {} Cesa-Bianchi%
}{%
Orabona%
\ \protect \BOthers {.}}{%
{\protect \APACyear {2015}}%
}]{%
orabona2015generalized}
\APACinsertmetastar {%
orabona2015generalized}%
\begin{APACrefauthors}%
Orabona, F.%
, Crammer, K.%
\BCBL {} Cesa-Bianchi, N.%
\end{APACrefauthors}%
\unskip\
\newblock
\APACrefYearMonthDay{2015}{}{}.
\newblock
{\BBOQ}\APACrefatitle {A generalized online mirror descent with applications to
  classification and regression} {A generalized online mirror descent with
  applications to classification and regression}.{\BBCQ}
\newblock
\APACjournalVolNumPages{Machine Learning}{99}{3}{411--435}.
\newblock

\newblock

\PrintBackRefs{\CurrentBib}

\bibitem [\protect \citeauthoryear {%
Orabona%
\ \BBA {} P{\'a}l%
}{%
Orabona%
\ \BBA {} P{\'a}l%
}{%
{\protect \APACyear {2018}}%
}]{%
orabona2018scale}
\APACinsertmetastar {%
orabona2018scale}%
\begin{APACrefauthors}%
Orabona, F.%
\BCBT {}\ \BBA {} P{\'a}l, D.%
\end{APACrefauthors}%
\unskip\
\newblock
\APACrefYearMonthDay{2018}{}{}.
\newblock
{\BBOQ}\APACrefatitle {Scale-free online learning} {Scale-free online
  learning}.{\BBCQ}
\newblock
\APACjournalVolNumPages{Theoretical Computer Science}{716}{}{50--69}.
\newblock

\newblock

\PrintBackRefs{\CurrentBib}

\bibitem [\protect \citeauthoryear {%
Ribeiro%
, Singh%
\BCBL {}\ \BBA {} Guestrin%
}{%
Ribeiro%
\ \protect \BOthers {.}}{%
{\protect \APACyear {2016}}%
}]{%
ribeiro2016should}
\APACinsertmetastar {%
ribeiro2016should}%
\begin{APACrefauthors}%
Ribeiro, M.T.%
, Singh, S.%
\BCBL {} Guestrin, C.%
\end{APACrefauthors}%
\unskip\
\newblock
\APACrefYearMonthDay{2016}{}{}.
\newblock
{\BBOQ}\APACrefatitle {" Why should i trust you?" Explaining the predictions of
  any classifier} {" why should i trust you?" explaining the predictions of any
  classifier}.{\BBCQ}
\newblock
 \APACrefbtitle {Proceedings of the 22nd ACM SIGKDD international conference on
  knowledge discovery and data mining} {Proceedings of the 22nd acm sigkdd
  international conference on knowledge discovery and data mining}\ (\BPGS\
  1135--1144).
\PrintBackRefs{\CurrentBib}

\bibitem [\protect \citeauthoryear {%
Song%
, Tekin%
\BCBL {}\ \BBA {} Van Der~Schaar%
}{%
Song%
\ \protect \BOthers {.}}{%
{\protect \APACyear {2014}}%
}]{%
song2014online}
\APACinsertmetastar {%
song2014online}%
\begin{APACrefauthors}%
Song, L.%
, Tekin, C.%
\BCBL {} Van Der~Schaar, M.%
\end{APACrefauthors}%
\unskip\
\newblock
\APACrefYearMonthDay{2014}{}{}.
\newblock
{\BBOQ}\APACrefatitle {Online learning in large-scale contextual recommender
  systems} {Online learning in large-scale contextual recommender
  systems}.{\BBCQ}
\newblock
\APACjournalVolNumPages{IEEE Transactions on Services
  Computing}{9}{3}{433--445}.
\newblock

\newblock

\PrintBackRefs{\CurrentBib}

\bibitem [\protect \citeauthoryear {%
Steinhardt%
\ \BBA {} Liang%
}{%
Steinhardt%
\ \BBA {} Liang%
}{%
{\protect \APACyear {2014}}%
}]{%
steinhardt2014adaptivity}
\APACinsertmetastar {%
steinhardt2014adaptivity}%
\begin{APACrefauthors}%
Steinhardt, J.%
\BCBT {}\ \BBA {} Liang, P.%
\end{APACrefauthors}%
\unskip\
\newblock
\APACrefYearMonthDay{2014}{}{}.
\newblock
{\BBOQ}\APACrefatitle {Adaptivity and optimism: An improved exponentiated
  gradient algorithm} {Adaptivity and optimism: An improved exponentiated
  gradient algorithm}.{\BBCQ}
\newblock
 \APACrefbtitle {International Conference on Machine Learning} {International
  conference on machine learning}\ (\BPGS\ 1593--1601).
\PrintBackRefs{\CurrentBib}

\bibitem [\protect \citeauthoryear {%
Warmuth%
}{%
Warmuth%
}{%
{\protect \APACyear {2007}}%
}]{%
warmuth2007winnowing}
\APACinsertmetastar {%
warmuth2007winnowing}%
\begin{APACrefauthors}%
Warmuth, M.K.%
\end{APACrefauthors}%
\unskip\
\newblock
\APACrefYearMonthDay{2007}{}{}.
\newblock
{\BBOQ}\APACrefatitle {Winnowing subspaces} {Winnowing subspaces}.{\BBCQ}
\newblock
 \APACrefbtitle {Proceedings of the 24th International Conference on Machine
  Learning} {Proceedings of the 24th international conference on machine
  learning}\ (\BPGS\ 999--1006).
\PrintBackRefs{\CurrentBib}

\bibitem [\protect \citeauthoryear {%
Xie%
, Bijral%
\BCBL {}\ \BBA {} Ferres%
}{%
Xie%
\ \protect \BOthers {.}}{%
{\protect \APACyear {2018}}%
}]{%
xie2018nonstop}
\APACinsertmetastar {%
xie2018nonstop}%
\begin{APACrefauthors}%
Xie, C.%
, Bijral, A.%
\BCBL {} Ferres, J.L.%
\end{APACrefauthors}%
\unskip\
\newblock
\APACrefYearMonthDay{2018}{}{}.
\newblock
{\BBOQ}\APACrefatitle {NonSTOP: A nonstationary online prediction method for
  time series} {Nonstop: A nonstationary online prediction method for time
  series}.{\BBCQ}
\newblock
\APACjournalVolNumPages{IEEE Signal Processing Letters}{25}{10}{1545--1549}.
\newblock

\newblock

\PrintBackRefs{\CurrentBib}

\end{thebibliography}
\end{document}